\documentclass[10pt]{amsart}
\usepackage{psfrag}
\usepackage[parfill]{parskip}
\usepackage{float, graphicx}
\usepackage[]{epsfig}
\usepackage{amsmath, amsthm, amssymb}
\usepackage{epsfig}
\usepackage{verbatim}
\usepackage{multicol}
\usepackage{url}
\usepackage{latexsym}
\usepackage{mathrsfs}
\usepackage[colorlinks, bookmarks=true]{hyperref}
\usepackage{graphicx}
\usepackage{enumerate}
\usepackage[top=2in, bottom=1.5in, left=1.1in, right=1.1in]{geometry}

\title{Partition statistics equidistributed with the number of hook difference one cells}
\author{Jiaoyang Huang, Andrew Senger, Peter Wear, and Tianqi Wu}

\address{Jiaoyang Huang, Department of Mathematics, Massachusetts Institute of Technology}
\email{jiaoyang@mit.edu}

\address{Andrew Senger, School of Mathematics, University of Minnesota}
\email{Andrew Senger: senge020@umn.edu}

\address{Peter Wear, Department of Mathematics, Massachusetts Institute of Technology}
\email{Peter Wear: pwear@mit.edu}

\address{Tianqi Wu, Department of Mathematics, Massachusetts Institute of Technology}
\email{Tianqi Wu: timwu@mit.edu}

\begin{document}
\maketitle
\newtheorem{theorem}{Theorem}[section]
\newtheorem*{theorem*}{Theorem}
\newtheorem{lemma}[theorem]{Lemma}
\newtheorem{remark}[theorem]{Remark}
\newtheorem{proposition}[theorem]{Proposition}
\newtheorem{conjecture}[theorem]{Conjecture}
\newtheorem{problem}[theorem]{Problem}
\newtheorem{corollary}[theorem]{Corollary}
\theoremstyle{definition}
\newtheorem{example}[theorem]{Example}
\newtheorem{definition}[theorem]{Definition}

\begin{abstract}
Let $\lambda$ be a partition, viewed as a Young diagram. We define the hook difference of a cell of $\lambda$ to be the difference of its leg and arm lengths. Define $h_{1,1}(\lambda)$ to be the number of cells of $\lambda$ with hook difference one. In \cite{BF}, algebraic geometry is used to prove a generating function identity which implies that $h_{1,1}$ is equidistributed with $a_2$, the largest part of a partition that appears at least twice, over the partitions of a given size. In this paper, we propose a refinement of the theorem of \cite{BF} and prove some partial results using combinatorial methods. We also obtain a new formula for the q-Catalan numbers which naturally leads us to define a new q,t-Catalan number with a simple combinatorial interpretation.
\end{abstract}

\section{Introduction}
In \cite{BF}, Buryak and Feigin proved the following partition identities using algebraic geometry:
\begin{theorem}\label{mainthm}
Let $\lambda=(\lambda_1,\lambda_2,\cdots,\lambda_k)$, $\lambda_1 \geq \lambda_2 \geq \cdots \geq \lambda_k$, be a partition and let $\mathcal{P}$ be the set of all partitions. View $\lambda$ as a Young diagram. Further define $h_{1,1} (\lambda)$ to be the number of cells of $\lambda$ which have leg length one greater than arm length. Say that these cells have \textbf{hook difference one}. Then
\begin{enumerate}

\item \label{mainthm1} (\cite{BF}, Theorems 1.2 and 1.3) \begin{align*} \sum_{\lambda \in \mathcal{P}} t^{h_{1,1} (\lambda)} q^{|\lambda|} = \sum_{\lambda\in \mathcal{P}}\prod_{i\geq 1}\left[\lambda_i-\lambda_{i+2}+1\atop \lambda_{i+1}-\lambda_{i+2}\right]_{t}q^{{\lambda_1\choose 2}+|\lambda|}.
\end{align*}

\item \label{mainthm2} (\cite{BF}, Theorem 1.3) \begin{align*}
\sum_{\lambda\in \mathcal{P}}\prod_{i\geq 1}\left[\lambda_i-\lambda_{i+2}+1 \atop \lambda_{i+1}-\lambda_{i+2}\right]_{t}q^{{\lambda_1\choose 2}+|\lambda|}=\prod_{i\geq 1}\frac{1}{(1-q^{2i-1})(1-tq^{2i})}.
\end{align*}
\end{enumerate}
\end{theorem}
In this paper, we attempt a combinatorial proof of these results. We give a combinatorial proof of (\ref{mainthm1}) in Section \ref{multisum_section}, as well as prove a generalization of this identity in Theorem \ref{multisumthm}. We prove a similar generalization of (\ref{mainthm2}) in the case where $t=1$ in Corollary \ref{multisumcor}. We also present some partial results towards a combinatorial proof of the full version of (\ref{mainthm2}) in Section \ref{2core_section}.

\subsection{$q$-Catalan Numbers}
We also explore a connection with the $q$-Catalan numbers. By restricting the partitions which we sum over in the left side of (\ref{mainthm1}) to those inside the $n$ by $n$ triangle in the upper left of the plane, we may exploit the duality between Young diagrams and Dyck paths to prove the following formula for the (Carlitz) $q$-Catalan numbers.

\begin{theorem}\label{thm1.2}
\begin{align*}
q^{{n}\choose {2}}C_n\left(\frac{1}{q}\right)=\sum_{\lambda=(\lambda_1,\lambda_2,\cdots,\lambda_k),\atop \lambda_1+k\leq n}\prod_{i \geq 1} {\lambda_i-\lambda_{i+2}+1 \choose \lambda_{i+1}-\lambda_{i+2}}q^{{\lambda_1\choose 2}+|\lambda|},
\end{align*}
where $C_n(q)$ is the $n$-th (Carlitz) $q$-Catalan number.
\end{theorem}

This leads us to define a new $q,t$-Catalan number in a natural way.

\begin{definition}
We define $C_n\left(\frac{1}{q},t\right)$ by
\begin{align*}
q^{-{{n}\choose {2}}}C_n\left(q,t\right)=\sum_{\lambda=(\lambda_1,\lambda_2,\cdots,\lambda_k),\atop \lambda_1+k\leq n}\prod_{i \geq 1} \left[\lambda_i-\lambda_{i+2}+1 \atop \lambda_{i+1}-\lambda_{i+2}\right]_t q^{-{\lambda_1\choose 2}-|\lambda|}.
\end{align*}
Replacing $q$ by $\frac{1}{q}$, we obtian the following more familiar expression,
\begin{align*}
q^{{n}\choose {2}}C_n\left(\frac{1}{q},t\right)=\sum_{\lambda=(\lambda_1,\lambda_2,\cdots,\lambda_k),\atop \lambda_1+k\leq n}\prod_{i \geq 1} \left[\lambda_i-\lambda_{i+2}+1 \atop \lambda_{i+1}-\lambda_{i+2}\right]_t q^{{\lambda_1\choose 2}+|\lambda|}.
\end{align*}
so that the coefficent of $t^i q^k$ is the number of partitions of size ${n \choose 2 } - k$ with $i$ hook difference one cells that fit inside the $n$ by $n$ triangle in the upper left of the plane.
\end{definition}
We discuss this in more detail in Section \ref{catalan_section}.

\subsection{$2$-cores and connections with other statistics on Young diagrams}
It is well known that
\begin{align*}
\prod_{i\geq 1}\frac{1}{(1-q^{2i-1})(1-tq^{2i})} = \sum_{\lambda \in \mathcal{P}} t^{a_2 (\lambda)} q^{|\lambda|},
\end{align*}
where $a_2 (\lambda)$ is the largest part of $\lambda$ that appears at least $2$ times. This allows us to interpret (\ref{mainthm1}) and (\ref{mainthm2}) as stating that $h_{1,1}$ and $a_2$ are equidistributed among the partitions of $n$. We provide a proof of a refinement of this statement in the case $h_{1,1} (\lambda) = a_2 (\lambda) = 0,1,2$ in Proposition \ref{k=0,1,2}.

Many of our results rely on the notion of the \textbf{$2$-core} of a partition, which will be defined in Section \ref{2core_section}. For now, we describe the $2$-core of a partition $\lambda$ as the partition which remains after removing all dominos from $\lambda$'s Young diagram after whose removal $\lambda$ remains a partition. It is a perhaps surprising fact that this is a well-defined process. It is easy to see that the possible $2$-cores are exactly the partitions of staircase shape $\{k,k-1,k-2,\cdots,3,2,1\}$, where $k$ is a non-negative integer.

\begin{example}
The 2-core of $\lambda=\{8,7,5,3,2,1\}$ is $\{4,3,2,1\}$.
\end{example}

In Section \ref{2core_section}, we prove the following refinement of the equidistribution of $h_{1,1}$ and $a_2$ in a special case.

\begin{theorem}\label{2corethm}
The statistics $h_{1,1}$ and $a_2$ are equidistributed on the set of partitions of $n$ with $2$-core size ${k+1 \choose 2}$ for all non-negative integers $n$ and $k$ such that $k \geq \frac{n-{k+1 \choose 2}}{2}$.
\end{theorem}

The condition in the theorem can be restated as saying that the largest part of the $2$-core of the partition is no smaller than the number of dominoes we needed to remove from the partition to obtain the $2$-core. We conjecture that this condition is in fact unnecessary.

\begin{conjecture}
The statistics $h_{1,1}$ and $a_2$ are equidistributed on the set of partitions of $n$ with $2$-core $\{k,k-1,\cdots, 1\}$ for all non-negative integers $n$ and $k$.
\end{conjecture}

This refines Theorem \ref{mainthm}. We propose a generalization of this as Conjecture \ref{mcoreconj}, and prove an analogous special case in Theorem \ref{bigcore}.

\subsection{Organization of the Paper}
In Section \ref{def_section}, we introduce some definitions and notation. In Section \ref{multisum_section}, we introduce a bijection between partitions and Eulerian tours of certain multigraphs which was presented in \cite{LW}, and use this to prove a generalization of (\ref{mainthm1}), as well as a generalization of (\ref{mainthm2}) for the $t=1$ case. In Section \ref{catalan_section}, we discuss a connection with the $q$-Catalan numbers. In Section \ref{2core_section}, $2$-cores and $2$-quotients of partitions are introduced and we prove Theorem \ref{2corethm} in addition to a host of other results. In Section \ref{mcore_section}, we generalize the results in Section \ref{2core_section} to the $m$-core case.

\subsection{Acknowledgements}
This research was conducted at the 2013 summer REU (Research Experience for Undergraduates) program at the University of Minnesota, Twin Cities, and was supported by NSF grants DMS-1001933, DMS-1067183, and DMS-1148634. The authors would like to thank Joel Lewis and Professors Gregg Musiker, Pavlo Pylyavskyy and Dennis Stanton, who directed the program, for their help. We would especially like to express the warmest thanks to our mentor Dennis Stanton, for his dedicated guidance throughout the research process. We would also like to thank Alex Csar for his help in editing this paper.

\section{Definitions}\label{def_section}
Throughout the paper, $\lambda$ and $\mu$ will denote partitions. We will view partitions either as a non-increasing sequence of non-negative integers $\lambda_1 \geq \lambda_2 \geq \lambda_3, \geq \cdots$, finitely many of which are nonzero, or a Young diagram in the English notation. Let $\lambda^T$ denote the conjugate partition of $\lambda$, $a(\lambda)$ denote the largest part of $\lambda$, and $l(\lambda)$ denote the number of nonzero parts of $\lambda$.

\subsection{Statistics on cells in Young diagrams} Given a cell $v$ in a Young diagram, we define the following:
\begin{itemize}
\item The \textbf{index} of $v$ is $(i, j)$ if $v$ is in row $i$ and column $j$. We often abuse the notation and say $v = (i, j)$.
\item The \textbf{arm} of $v$ is the set of cells in its row to its right; its cardinality is the \textbf{arm length}, denoted by $a_v$.
\item The \textbf{leg} of $v$ is the set of cells in its column below it; its cardinality is the \textbf{leg length}, denoted by $l_v$.
\item The \textbf{hook} of $v$ is the union of $v$ and its arm and leg; its cardinality is the \textbf{hook length} of $v$, which equals $a_v+l_v+1$.
\item Let $\alpha,\beta \in \mathbb{Z}$. The \textbf{$(\alpha, \beta)$-label} of $v = (i, j)$ is $\alpha i + \beta j$. (The modifier $(\alpha, \beta)$ is dropped if it is clear in context.)
\end{itemize}

\subsection{Statistics on partitions} Given a partition $\lambda$, we define the following:
\begin{itemize}
\item The \textbf{size} $|\lambda|$ is the number of cells in $\lambda$, or equivalently the sum of the parts of $\lambda$.
\item $a_m(\lambda)$ is the size of the largest part in $\lambda$ with multiplicity at least $m$. We say that $\lambda$ is \textbf{$m$-regular} if $a_m(\lambda) = 0$, and say that $\lambda$ is \textbf{$m$-restricted} if $a_m(\lambda^T) = 0$.
\item The $k$-th \textbf{$(\alpha, \beta)$-diagonal} of $\lambda$ is the set of cells in $\lambda$ whose $(\alpha, \beta)$-label is $k$; its cardinality is the $k$-th \textbf{$(\alpha, \beta)$-diagonal length}, denoted by $d^{(\alpha, \beta)}_{k}(\lambda)$. We call the sequence $\{d^{(\alpha, \beta)}_{k}(\lambda)\}$ the \textbf{$(\alpha, \beta)$-diagonal pattern} of $\lambda$. (The modifier $(\alpha, \beta)$ is dropped if it is clear in context.) In the case $\alpha = m-1, \beta = 1$, we change the modifier $(\alpha, \beta)$ to $m$.
\item Let $\alpha,\beta$ be non-negative integers, not both zero. Define $H_{\alpha, \beta}(\lambda)=\{v \in \lambda |\alpha l_v =\beta(a_v+1), \quad (\alpha+\beta)|a_v+l_v+1\}$ and $h_{\alpha, \beta}(\mu)=|H_{\alpha, \beta}(\mu)|$. Note that this generalizes our earlier definition of $h_{1,1}(\mu)$ as the number of hook difference one cells of $\lambda$, which are cells with leg and arm difference one.
\end{itemize}

Denote $\mathcal{P}$ the set of all partitions. Given a nonnegative integer valued statistic on partitions $f: \mathcal{P} \rightarrow \mathbb{N}$ and a set $S \subset \mathcal{P}$, we call the sequence $c_i(f, S) = |\{\lambda \in S | f(\lambda) = i\}|$ the \textbf{$f$-distribution of $S$}. We say that $f$ and $g$ are \textbf{equidistributed} over $S$ if $c_i(f, S) = c_i(g, S)$ for all $i \in \mathbb{N}$, and we say that $f$ is \textbf{identically distributed} over a family of sets $\mathcal{S}$ if $c_i(f, S_1) = c_i(f, S_2)$ for any $S_1, S_2 \in \mathcal{S}$ and $i \in \mathbb{N}$.

Given a statement $P$, we define $\chi(P) = 1$ if $P$ is true and $\chi(P) = 0$ if $P$ is false.

Given a power series $Q(x)$ and a non-negative integer $i$, we let $[x^i] Q(x)$ denote the coefficient of $x^i$ in $Q(x)$.

\section{A Multisum Formula for Partitions}\label{multisum_section}

Define an equivalence relation on $\mathcal{P}$ by $\lambda \sim \mu$ whenever $\lambda$ and $\mu$ have the same $m$-diagonal pattern, which we call the $m$-diagonal equivalence relation. In this section, we compute the generating function of the statistic $h_{m-1,1}$ for any given $m$-diagonal equivalence class as a product of $t$-binomial coefficients, which yields a proof of the following generalization of Theorem \ref{mainthm} (\ref{mainthm1}).

\begin{theorem}\label{multisumthm}
\begin{align*}
\sum_{\mu\in\mathcal{P}}t^{h_{m-1,1}(\mu)}q^{|\mu|}
=\sum_{\{s_i\} \in \mathcal{S}_m}
q^{(m-1)\binom{s_1}{2} + \sum s_i}
\prod_{j\geq 1}
\left[s_j - s_{j+m} + \chi (m-1 | j) \atop s_{j+1} - s_{j+m}\right]_t,
\end{align*}
where $\mathcal{S}_m$ is the set of sequences of nonnegative integers $\{s_i\}_{i > 0}$, finitely many of which are nonzero, such that
\begin{itemize}
\item $s_i \geq s_{i+1} - \chi (m-1 | i)$,
\item $s_i \geq s_{i+(m-1)}$.
\end{itemize}
Note that these conditions may be dropped if we assume a $t$-binomial coefficient with a negative entry to be zero. Also note that $\mathcal{S}_2 = \mathcal{P}$, so setting $m=2$ gives Theorem \ref{mainthm} (\ref{mainthm1}).
\end{theorem}

If we take $t=1$ in the above theorem, we obtain a generalization of Theorem \ref{mainthm} (\ref{mainthm2}) in the $t=1$ case.
\begin{corollary}\label{multisumcor}
\begin{align*}
\frac{1}{\prod_{i\geq 1}(1-q^i)}
=\sum_{\{s_i\} \in \mathcal{S}_m}
q^{(m-1)\binom{s_1}{2} + \sum s_i}
\prod_{j\geq 1}
\binom{s_j - s_{j+m} + \chi (m-1 | j)} {s_{j+1} - s_{j+m} },
\end{align*}
\end{corollary}

\begin{remark}
For the special case $m=2$, the above formula gives the number of Young diagrams of a $2$-diagonal equivalence class with diagonal pattern $\{d_i\}$. It is interesting to notice that these $2$-diagonal equivalence classes are the same as rook equivalence classes, and each mutiproduct in the above formula gives the number of Ferrers boards in corresponding rook equivalent class. Some references on rook sequence are \cite{GJW} and \cite{RS}.
\end{remark}

Using the same method, we also give a partial result towards proving Theorem \ref{mainthm} (\ref{mainthm2}).
\begin{theorem}\label{k=0,1,2}
For $r = 0, 1, 2$,
\begin{align*}
[t^r] \sum_{\lambda\in \mathcal{P}}\prod_{i\geq 1}\left[\lambda_i-\lambda_{i+2}+1 \atop \lambda_{i+1}-\lambda_{i+2}\right]_{t}q^{{\lambda_1\choose 2}+|\lambda|}= [t^r] \prod_{i\geq 1}\frac{1}{(1-q^{2i-1})(1-tq^{2i})}.
\end{align*}
\end{theorem}

To prove these results, we begin by describing some bijections between partitions and other combinatorial objects that were used in \cite{LW}.

\subsection{Border paths}
Fix coprime positive integers $\alpha, \beta$ such that $\alpha + \beta = m$. The $k$-th \textbf{level} is the line $\beta x -\alpha y = k$ in the coordinate plane. Let $R_n$ be the $\alpha n \times \beta n$ rectangle with vertices $(0,0)$, $(0, -\beta n)$, $(\alpha n, 0), (\alpha n, -\beta n)$. Let $\Delta_n$ be the triangle with vertices $(0,0)$, $(0, -\beta n)$, $(\alpha n, 0)$; its hypotenuse lies on $\alpha \beta n$-th level.

Place $\lambda$ in the coordinate plane so that the top-left vertex of its top-left cell is at the origin. Then the bottom-right vertex of a cell on the $k$-th $(\alpha, \beta)$-diagonal of $\lambda$ lies on the $k$-th level. Thus, $\lambda$ is contained in $\Delta_n$ if and only if $\alpha \beta n \geq k_\lambda$, where $k_\lambda$ is the maximal $(\alpha, \beta)$-label in $\lambda$.

Define the \textbf{order} $n_\lambda$ to be the smallest integer $n$ such that $\lambda$ is contained in $\Delta_n$, then $n_\lambda = \lceil k_\lambda /\alpha \beta\rceil$. Define the \textbf{border path} $Bdp(\lambda)$ to be the staircase walk in $R_{n_\lambda}$ that traces the shape of $\lambda$. By the choice of $n_\lambda$, $Bdp(\lambda)$ stays in $\Delta_{n_\lambda}$ but not in $\Delta_{n_\lambda - 1}$. Clearly, partitions are identified by their border paths.

\begin{example}\label{borderex}
Take $\alpha=3$ and $\beta=1$. The following is the border path of $\lambda=\{16,6,6,6,5\}$. In this example $k_\lambda =17$ and $n_\lambda = 7$, so we place $\lambda$ in a $21\times 7$ triangle.
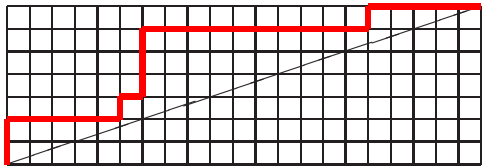
\begin{figure}[H]
\begin{center}
\begin{picture}(21,10)

\multiput(0,0)(1,0){21}{\line(1,0){1}}
\multiput(0,0)(1,0){22}{\line(0,1){1}}
\multiput(0,1)(1,0){21}{\line(1,0){1}}
\multiput(0,1)(1,0){22}{\line(0,1){1}}
\multiput(0,2)(1,0){21}{\line(1,0){1}}
\multiput(0,2)(1,0){22}{\line(0,1){1}}
\multiput(0,3)(1,0){21}{\line(1,0){1}}
\multiput(0,3)(1,0){22}{\line(0,1){1}}
\multiput(0,4)(1,0){21}{\line(1,0){1}}
\multiput(0,4)(1,0){22}{\line(0,1){1}}
\multiput(0,5)(1,0){21}{\line(1,0){1}}
\multiput(0,5)(1,0){22}{\line(0,1){1}}
\multiput(0,6)(1,0){21}{\line(1,0){1}}
\multiput(0,6)(1,0){22}{\line(0,1){1}}
\multiput(0,7)(1,0){21}{\line(1,0){1}}

\linethickness{1pt}
\multiput(0,0)(3,1){7}{\line(3,1){3}}
\multiput(0,0)(1,1){8}{\line(1,1){1}}
{\color{red}
\linethickness{2pt}
\multiput(0,2)(1,0){5}{\line(1,0){1}}
\put(5,3){\line(1,0){1}}
\multiput(6,6)(1,0){10}{\line(1,0){1}}
\multiput(16,7)(1,0){5}{\line(1,0){1}}
\put(0,0){\line(0,1){1}}
\put(0,1){\line(0,1){1}}
\put(5,2){\line(0,1){1}}
\put(6,3){\line(0,1){1}}
\put(6,4){\line(0,1){1}}
\put(6,5){\line(0,1){1}}
\put(16,6){\line(0,1){1}}
}
\end{picture}
\end{center}
\caption{The border path of $\lambda$.}\label{graph1}
\end{figure}
\end{example}

\subsection{Eulerian tours of multigraphs}
Using the border path of $\lambda$, define a directed multigraph $M(\lambda)$, called the multigraph of $\lambda$, as follows: the vertex set $V(\lambda)$ consists of $i$ such that $Bdp(\lambda)$ intersects $(\alpha \beta n_\lambda - i)$-th level, and there is an edge from $i$ to $j$ for each step in $Bdp(\lambda)$ moving from $(\alpha \beta n_\lambda - i)$-th level to $(\alpha \beta n_\lambda - j)$-th level. Then $Bdp(\lambda)$, thus $\lambda$, can be naturally interpreted as an Eulerian tour $T(\lambda)$ on $M(\lambda)$ that begins and ends at $0$, from which we can recover $\lambda$.

For a directed multigraph $M$, we call an edge $e$ in $M$ a \textbf{north edge} if $e$ goes from some $i$ to $i+\alpha$ and an \textbf{east edge} if $e$ goes from some $i$ to $i-\beta$. Note that $M(\lambda)$ satisfy the following conditions:
\begin{itemize}
\item The vertex set $V(\lambda)$ is a finite subset of nonnegative integers containing $0$,
\item The edge multiset of $M(\lambda)$ consist of $\alpha n_\lambda$ east edges and $\beta n_\lambda$ north edges,
\item $M(\lambda)$ is \textit{balanced}, i.e. at each vertex there are as many incoming edges as outgoing edges,
\item $M(\lambda)$ is \textit{connected}, i.e. there is a path from any vertex to any other vertex,
\item There are at least $1 + \chi (\alpha | i)$ north edges leaving $i$ for some $0 \leq i < \alpha \beta$.
\end{itemize}
We call a multigraph that satisfies the above conditions a \textit{valid} multigraph of order $n_\lambda$. Note that the third and fourth conditions ensure the existence of an Eulerian tour. Given a valid multigraph $M$ of order $n$, an Eulerian tour $T$ on $M$ that begins and ends at $0$ can be naturally interpreted as a staircase walk in $R_n$ (by the second condition) that stays in $\Delta_n$ (by the first condition) but not in $\Delta_{n-1}$ (by the last condition), thus the border path of a partition $\mu$ of order $n$, from which we can recover $T$. Thus, partitions of order $n$ correspond to Eulerian tours on valid multigraphs of order $n$ that begin and end at $0$.

\subsection{Departure words and north patterns}
For a valid multigraph $M$, let $N_i(M)$ (resp. $E_i(M)$) to be the number of north edges (resp. east edges) leaving $i$ if $i$ is a vertex of $M$ and $0$ otherwise. Clearly, $\{N_i(M)\}, \{E_i(M)\}$ are uniquely determined by $M$ and vice versa. By abuse of notation, we define the \textbf{north pattern} $\{N_i(\lambda)\}$ (resp. \textbf{east pattern} $\{E_i(\lambda)\}$) of a partition $\lambda$ to be the sequence $\{N_i(M(\lambda))\}$ (resp. $\{E_i(M(\lambda))\}$).

Let $W(N^aE^b)$ denote the set of binary words on $a$ letters $N$ and $b$ letters $E$. The Eulerian tour $T(\lambda)$ is made more explicit by constructing a sequence of binary words $\{w^{i}(\lambda) \in W(N^{N_i(\lambda)}E^{E_i(\lambda)})\}$, called the \textbf{departure words} of $\lambda$, as follows. Starting with a sequence of empty words, we traverse the tour; every time the tour visits $i$, we append to the $i$-th word a letter $N$ if the next edge is a north edge and a letter $E$ if the next edge is an east edge. After the tour is completed, $w^{i}(\lambda)$ is defined as the $i$-th departure word. It is easy to see that $T(\lambda)$, and thus $\lambda$, can be recovered from $\{w^{i}(\lambda)\}$.

\begin{example}
The Young diagram of Example \ref{borderex} has to departure words $w^0=N$, $w^1=NE$, $w^2=ENE$, $w^3=NENEE$, $w^4=EEEE$, $w^5=EEE$, $w^6=ENE$, $w^7=E$, $w^8=E$, $w^9=NE$, $w^{10}=E$, $w^{11}=E$, $w^{12}=E$.
\end{example}

It is easy to compute the size $|\lambda|$, the diagonal pattern $\{d^{\alpha, \beta}_k(\lambda)\}$, and $h_{\alpha,\beta}(\lambda)$ from the sequence of departure words of $\lambda$, as we now show.

\begin{proposition}\label{area}
For any partition $\lambda$ with departure words $\{w^{i}(\lambda)\in W(N^{N_i}E^{E_i})\}$, we have
\begin{align*}|\lambda| = \alpha \beta \binom{n_{\lambda}}{2} +n_{\lambda} \sum_{j = 0}^{\beta - 1} \lfloor \alpha j / \beta \rfloor - \sum_{i} \lfloor i/\beta \rfloor N_i
\end{align*}
If $\alpha = m-1, \beta = 1$, this simplifies to
\begin{align*}
|\lambda|=(m-1){\sum N_i \choose 2}-\sum_{i\geq 0}i N_i
\end{align*}
\end{proposition}
\begin{proof}
We compute $|\lambda|$ in two steps. We first compute the size of the maximal partition $P$ that fits in $\Delta_{n_\lambda}$, then we compute the size difference between $P$ and $\lambda$. We have $P_i = \alpha n_\lambda - \lceil \alpha i/\beta \rceil$, so $|P| = \alpha \beta \binom{n_{\lambda}}{2} + n_\lambda \sum_{j = 0}^{\beta - 1} \lfloor \alpha j / \beta \rfloor$. It remains to show $|P| - |\lambda| = \sum_{i = 0}^l \lfloor i/\beta \rfloor N_i$. For each north step leaving $(\alpha \beta n - i)$-th level, the number of cells to its right in $P$ is $\lfloor i/\beta \rfloor$. Summing up these contributions to $|P| - |\lambda|$ over all north steps gives the identity.
\end{proof}

\begin{proposition}\label{ND}
For any partition $\lambda$ with departure words $\{w^{i}(\lambda)\in W(N^{N_i}E^{E_i})\}$, we have
\begin{align*}
d^{(\alpha, \beta)}_k(\lambda) &= \begin{cases} f_{\alpha, \beta}(k) - \displaystyle\sum_{v < k , \beta | k - v} N_{\alpha \beta n_\lambda - v} & \text{if } 0< k \leq \alpha \beta n_\lambda \\
0 & \text{else} \end{cases} \\
N_v &= \begin{cases} d^{(\alpha, \beta)}_{\alpha\beta n_\lambda - v}(\lambda) - d^{(\alpha, \beta)}_{\alpha\beta n_\lambda - v + \beta}(\lambda) + \chi (\alpha | v) & \text{if } 0\leq v < \alpha\beta n_\lambda \\
0 & \text{else} \end{cases}
\end{align*}
where $f_{\alpha, \beta}(k)$ is the number of positive integer solutions to the equation $\alpha i + \beta j = k$. If $\alpha = m-1, \beta = 1$, these simplify to
\begin{align*}
d^{m}_k(\lambda) &= \begin{cases} \lfloor k/(m-1) \rfloor - \displaystyle\sum_{v < k} N_{(m-1) n_\lambda - v} & \text{if } 0< k \leq (m-1) n_\lambda \\
0 & \text{else}\end{cases}\\
N_v &= \begin{cases} d^{m}_{(m-1) n_\lambda - v}(\lambda) - d^{m}_{(m-1)n_\lambda - v + 1}(\lambda) + \chi (\alpha | v) & \text{if } 0\leq v < (m-1) n_\lambda \\
0 & \text{else} \end{cases}
\end{align*}
Thus, the north pattern determines the diagonal pattern and vice versa.
\end{proposition}
\begin{proof} The formula for diagonal pattern follows from an argument analogous to the proof of $\ref{area}$. The formula for north pattern is obtained from the former by computing differences.
\end{proof}

\begin{proposition}\label{hookpro}
For a partition $\lambda$ with departure words $\{w^{i}(\lambda)\}$, we have
\begin{align*}
h_{\alpha,\beta}(\lambda)=\sum_{i\geq0} \operatorname{inv}(w^{i}),
\end{align*}
where $\operatorname{inv}(w)$ is the total number of inverse pairs $EN$ in the word $w$.
\end{proposition}
\begin{proof}
Given a cell $v \in \lambda$, we call the rightmost cell in its arm the \textbf{hand} of $v$, and the bottommost cell in its leg the \textbf{foot} of $v$. For $v\in H_{\alpha,\beta}(\lambda)$, we have $\alpha l_v=\beta (a_v+1)$. If $v$ has index $(i, j)$, then its hand has index $(i, j+a_v)$ and lies on the $(\alpha i + \beta (j + a_v))$-th diagonal, and its foot has index $(i+l_v, j)$ and lies on the $(\alpha (i+l_v) + \beta j)$-th diagonal. Therefore the bottom-right vertex of $v$'s hand lies on $(\alpha \beta n_\lambda - \alpha i- \beta (j + a_v))$-th level and the bottom-left vertex of $v$'s foot lies on $(\alpha \beta n_\lambda - \alpha (i+l_v) - \beta (j+1))$-th level. Since $\alpha l_v=\beta (a_v+1)$, the two vertices lie on the same level, and the steps leaving them correspond to a pair $E$ and $N$ in the departure word corresponding to that level.

Therefore each cell $v\in H_{\alpha,\beta}(\lambda)$ determines a unique inverse pair in some departure word $w^i$. It is easy to see that the above process is reversible, so it defines a bijection between $H_{\alpha, \beta}(\lambda)$ and inverse pairs $EN$ in the departure words $\{w^{i}(\lambda)\}$.
\end{proof}

\begin{example}
In the following graph, the red edges are the border of the partition $\lambda$, the horizontal one corresponds a departure letter $E$, and the vertical one corresponds to a departure letter $N$. And the cell $v$ is in $H_{3,1}$.

\setlength{\unitlength}{0.5cm}
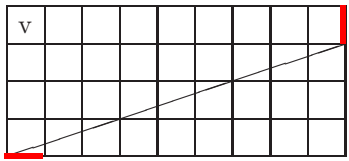
\begin{figure}[H]
\begin{center}
\begin{picture}(10,4)
\put(0.3,3.3){v}
\multiput(0,0)(1,0){9}{\line(1,0){1}}
\multiput(0,0)(1,0){10}{\line(0,1){1}}
\multiput(0,1)(1,0){9}{\line(1,0){1}}
\multiput(0,1)(1,0){10}{\line(0,1){1}}
\multiput(0,2)(1,0){9}{\line(1,0){1}}
\multiput(0,2)(1,0){10}{\line(0,1){1}}
\multiput(0,3)(1,0){9}{\line(1,0){1}}
\multiput(0,3)(1,0){10}{\line(0,1){1}}
\multiput(0,4)(1,0){9}{\line(1,0){1}}

\multiput(0,0)(3,1){3}{\line(3,1){3}}
\linethickness{2pt}
{\color{red}
\put(-0.3,0){\line(1,0){1}}
\put(8.7,3){\line(0,1){1}}
}
\end{picture}
\end{center}
\caption{A cell $v\in H_{3,1}(\lambda)$.}
\end{figure}
\end{example}

For the rest of this section, we will fix $(\alpha,\beta)=(m-1,1)$, where $m\geq 2$. In this case, it is easy to characterize the set of valid north (east) patterns, diagonal patterns, and sequences of departure words, as we now do.

\begin{proposition}\label{north}
A sequence $\{N_i\}$ is the north pattern of a partition if and only if
\begin{itemize}
\item $N_i = 0$ for $i < 0$ and for $i > l$, for some $l$,
\item $\sum_{j=1}^{m-1}N_{i-j} > 0$ for $1\leq i\leq l$,
\item $\max\{N_0-1,N_1,N_2,\cdots,N_{m-2}\}\geq 1$.
\end{itemize}
The corresponding east pattern $\{E_i\}$ is unique, given by
\begin{align*}
E_i(\lambda) = \sum_{j=1}^{m-1} N_{i-j}(\lambda).
\end{align*}
\end{proposition}
\begin{proof}
The conditions for a multigraph $M$ to be valid translates to the following conditions on $\{N_i(M)\}, \{E_i(M)\}$:
\begin{itemize}
\item $N_i(M) = E_i(M) = 0$ for $i < 0$ and $i > l$ for some $l$,
\item $N_{i-(m-1)}(M) + E_{i+1}(M) = N_i(M) + E_i(M)$ for all $i$,
\item $E_1(M), E_2(M), \cdots, E_{l-1}(M) \geq 1$,
\item $\max(N_0(M)-1, N_1(M), N_2(M), \cdots, N_{m-2}(M)) \geq 1$.
\end{itemize}
The first two conditions give the equation for $\{E_i\}$ in terms of $\{N_i\}$, and the rest correspond to the conditions on $\{N_i\}$.
\end{proof}

\begin{proposition}\label{diag}
A sequence of nonnegative integers $\{d_i\}_{i > 0}$ is the diagonal pattern of some partition if and only if for some $\{s_i\}_{i > 0} \in \mathcal{S}_m$, $\{d_i\}$ begins with $m-1$ terms equal to 0, $m-1$ terms equal to 1, ..., $m-1$ terms equal to $s_1-1$, followed by $\{s_i\}$. Thus, diagonal patterns are naturally represented by sequences in $\mathcal{S}_m$.
\end{proposition}
\begin{proof} This follows from Proposition \ref{north} after applying the formula in Proposition \ref{ND}.
\end{proof}

\begin{proposition}\label{dep}
A sequence of words $w^{i}\in W(N^{N_i}E^{E_i})$ is a sequence of departure words if and only if $\{N_i\}$ is a north pattern, $\{E_i\}$ its corresponding east pattern, and for $i\geq 1$, every nonempty $w^{i}$ ends with the letter $E$.
\end{proposition}
\begin{proof}
Let $M$ be the multigraph with $N_i(M) = N_i$ and $E_i(M) = E_i$. Let $\operatorname{TreeD}(M)$ be the set of all oriented spanning trees leading to the root $0$ in $M$. By Theorem 14 in \cite{LW}, $\{w^i\}$ is a sequence of departure words if and only if $M$ is valid and for $T \in \operatorname{TreeD}(M)$, $w^i$ ends with the letter corresponding to the the edge leaving $i$ in $T$. In our case, $\alpha = m-1, \beta = 1$, $\operatorname{TreeD}(M)$ consists of a single tree $T$, whose edges are all east edges.
\end{proof}

\subsection{Proofs of Theorems \ref{multisumthm} and \ref{k=0,1,2}}
Let $D$ be a diagonal equivalence class with diagonal pattern $\{d_i\}$, north pattern $\{N_i\}$, and east pattern $\{E_i\}$; the relations between these sequences are given in Propositions \ref{ND} and \ref{north}. By Proposition \ref{dep}, the number of sequences of departure words $\{w^i\}$ such that $w^i \in W(N^{N_i}E^{E_i})$ is
\begin{align*}
\prod_{j\geq 1}{E_j + N_j - 1\choose N_{j}} = \prod_{j\geq 1}{\sum_{i=0}^{m-1}N_{j-i}-1\choose N_{j}} = \prod_{j\geq 1}{s_{j} - s_{j+1} + \chi (m-1 | j) \choose s_{j+1} - s_{j+m}}.
\end{align*}
where $\{s_i\}$ is the sequence in $\mathcal{S}_m$ representing the diagonal pattern of $D$ (as in Proposition \ref{diag}). By Proposition \ref{hookpro}, if we replace the binomial coefficients above by $t$-binomial coefficients, this is the generating function of $h_{m-1,1}$ on $D$, as it is well known that $t$-binomial coefficients count words by number of inversions. Putting in powers of $q$ for the size of partition, which is just the sum of all terms in the diagonal pattern, and summing over diagonal equivalence classes, we get Theorem \ref{multisumthm}.

From the above proof, it is worth stating separately the following generating formula of $h_{m-1,1}$ on any $m$-diagonal equivalent class.
\begin{corollary}{\label{generatingfunction}}
Let $D$ be an $m$-diagonal equivalent class with diagonal pattern $\{d_i\}$, then we have the generating function of $h_{m-1,1}$ on $D$:
\begin{align*}
\sum_{\mu\in D}t^{h_{m-1,1}(\mu)}=\prod_{j\geq 1}\left[s_j - s_{j+m} + \chi (m-1 | j) \atop s_{j+1} - s_{j+m}\right]_t,
\end{align*}
where $\{s_i\}$ is the sequence in $\mathcal{S}_m$ representing the diagonal pattern of $D$.
\end{corollary}

For the rest of this section we fix $m = 2$. Theorem \ref{k=0,1,2} follows from Theorem \ref{multisumthm} and the proposition below.

\begin{proposition}\label{k=0,1,2}
Let $D$ be a 2-diagonal equivalence class. For $r = 0, 1, 2$, we have
\begin{align*}
|\{\mu| \mu \in D, h_{1,1}(\mu)=r\}| = |\{\mu| \mu \in D, a_2(\mu)=r\}|
\end{align*}
\end{proposition}
\begin{proof} Given a partition $\mu$, we construct its \textit{initial} words by running the same algorithm for constructing the departure words, except stopping at the first step corresponding to the largest repeated part of $\mu$, and removing any trailing $N$'s from each word. We call the sequence of their complements in the departure words the \textit{remaining} words of $\mu$. Clearly the remaining words have no inversions, as they represent a partition with distinct parts.

For $r = 0$, the two sets in the Proposition are identical.

For $r = 1$, $h_{1,1}(\mu) = 1$ if and only if all of its departure words have no inversions except one, which is of the form $N\cdots NENE\cdots E$, and $a_2(\mu)$ is $1$ if and only if its initial words are all empty except one, which is $E$. Clearly, the two sets are in bijection.

For $r = 2$, $h_{1,1}(\mu) = 2$ if and only if its departure words
\begin{enumerate}[(a)]
\item \label{itm1}have no inversions except two, which are of the form $N\cdots NENE\cdots E$.
\item \label{itm2}have no inversions except one, which is of the form $N\cdots NEENE\cdots E$.
\item \label{itm3}have no inversions except one, which is of the form $N\cdots NENNE\cdots E$.
\end{enumerate}
and $a_2(\mu)$ is $2$ if and only if its initial words are
\begin{enumerate}[(A)]
\item \label{itma}all empty except two, which are $E$.
\item \label{itmb}all empty except one, which is $EE$.
\item \label{itmc}all empty except two consecutive words, which are $NE, E$ (in that order).
\end{enumerate}
Clearly, (\ref{itm1}) and (\ref{itma}), (\ref{itm2}) and (\ref{itmb}) are in bijection. To see that (\ref{itm3}) and (\ref{itmc}) are in bijection, note that for $\mu$ of type (\ref{itm3}), we may move the inversion sequence $ENN$ to the beginning of the word, change it to $NEN$, and move one letter $E$ in the next depature word to its beginning. This gives a partition of type (\ref{itmc}). This process is reversible, and thus is a bijection.
\end{proof}
\begin{remark}Proposition \ref{k=0,1,2} does not generalize to $r \geq 3$.
\end{remark}

\section{The connection to $q$-Catalan numbers}\label{catalan_section}

\subsection{A $q$-Catalan identity}By restricting the sum in the left side of Theorem \ref{multisumthm} to partitions that fit in the upper triangular region with side length $n$, and restricting the right side accordingly, we obtain a similar identity to Corollary \ref{multisumcor} for the $q$-Catalan numbers.

We begin by noting that a partition with $2$-diagonal pattern $\{1,2,\cdots, \lambda_1-1, \lambda_1,\lambda_2,\cdots \}$ is in the $n$ by $n$ upper triangular region if and only if $\lambda_1+l(\lambda)\leq n$. Therefore Theorem \ref{multisumthm} implies that the generating function for partitions in this region is
\begin{align}
\label{resn}f_n(q)=\sum_{\lambda \in \mathcal{P},\atop \lambda_1+l(\lambda)\leq n}\prod {\lambda_i-\lambda_{i+2}+1 \choose \lambda_{i+1}-\lambda_{i+2}}q^{{\lambda_1\choose 2}+|\lambda|}
\end{align}

On the other hand, each partition in this region corresponds to a word with $n$ $E$'s and $n$ $N$'s arranged in such a way that the number of $E$'s never exceeds the number of $N$'s. We call such a word a Dyck word of length $n$, and the corresponding lattice path a Dyck path. Let $P_n$ denote the set of Dyck paths of length $n$.

\begin{example}
The following graph is a Dyck path of length $n=8$, with coarea $12$.
\begin{figure}[H]
\begin{center}
\begin{picture}(10,10)
\linethickness{0.5pt}

\multiput(0,0)(1,0){8}{\line(1,0){1}}
\multiput(0,0)(1,0){9}{\line(0,1){1}}
\multiput(0,1)(1,0){8}{\line(1,0){1}}
\multiput(0,1)(1,0){9}{\line(0,1){1}}
\multiput(0,2)(1,0){8}{\line(1,0){1}}
\multiput(0,2)(1,0){9}{\line(0,1){1}}
\multiput(0,3)(1,0){8}{\line(1,0){1}}
\multiput(0,3)(1,0){9}{\line(0,1){1}}
\multiput(0,4)(1,0){8}{\line(1,0){1}}
\multiput(0,4)(1,0){9}{\line(0,1){1}}
\multiput(0,5)(1,0){8}{\line(1,0){1}}
\multiput(0,5)(1,0){9}{\line(0,1){1}}
\multiput(0,6)(1,0){8}{\line(1,0){1}}
\multiput(0,6)(1,0){9}{\line(0,1){1}}
\multiput(0,7)(1,0){8}{\line(1,0){1}}
\multiput(0,7)(1,0){9}{\line(0,1){1}}
\multiput(0,8)(1,0){8}{\line(1,0){1}}

\linethickness{2pt}
\multiput(0,0)(1,1){8}{\line(1,1){1}}
{\color{red}
\put(0,2){\line(1,0){1}}
\put(1,3){\line(1,0){1}}
\put(2,6){\line(1,0){1}}
\put(3,6){\line(1,0){1}}
\put(4,7){\line(1,0){1}}
\put(5,8){\line(1,0){1}}
\put(6,8){\line(1,0){1}}
\put(7,8){\line(1,0){1}}

\put(0,0){\line(0,1){1}}
\put(0,1){\line(0,1){1}}
\put(1,2){\line(0,1){1}}
\put(2,3){\line(0,1){1}}
\put(2,4){\line(0,1){1}}
\put(2,5){\line(0,1){1}}
\put(4,6){\line(0,1){1}}
\put(5,7){\line(0,1){1}}
}
\end{picture}
\end{center}
\caption{A Dyck path of length $n=8$.}
\end{figure}
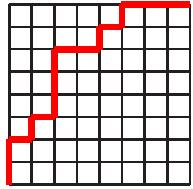

\end{example}

Given any path $p\in P_n$, define
\begin{align*}
\text{coarea}(p)=\#\{\text{cells between }p \text{ and the diagonal } y=x\}
\end{align*}
Then we may define the $n$-th (Carlitz) $q$-Catalan number by
\begin{align*}
C_n(q)=\sum_{p\in P_n}q^{\text{coarea}(p)},\quad n=1,2,3,4\cdots.
\end{align*}
By convention we let $C_0=1$.
See \cite{FH} for further discussion of the $q$-Catalan numbers.

It is now easy to see that $f_n$ and $C_n$ satisfy the following relation:
\begin{align*}
f_n(q)=q^{{n}\choose {2}}C_n\left(\frac{1}{q}\right).
\end{align*}
Plugging in (\ref{resn}), we obtain Theorem \ref{thm1.2}, which recall states that
\begin{align*}
q^{{n}\choose {2}}C_n\left(\frac{1}{q}\right)=\sum_{\lambda=(\lambda_1,\lambda_2,\cdots,\lambda_k),\atop \lambda_1+k\leq n}\prod_{i \geq 1} {\lambda_i-\lambda_{i+2}+1 \choose \lambda_{i+1}-\lambda_{i+2}}q^{{\lambda_1\choose 2}+|\lambda|}.
\end{align*}

We now provide an alternative proof of Theorem \ref{thm1.2} using the $q$-Vandermonde theorem, from it we obtain a convolution-like recurrence for the generating function of $h_{1,1}(\mu)$ in each 2-diagonal equivalence class.

Given any partition in the $n$ by $n$ upper trianglular region, where we assume the origin to be the bottom left of the region, define $s_i$ to be the number of cells on the diagonal $y=x+i$, which are not in the partition. It is easy to see that the sequence $\{s_i\}$ first decreases to $0$ then remains at $0$:
\begin{align}
\label{decreasing}n=s_0>s_1>s_2>\cdots>s_{l}>s_{l+1}=s_{l+2}=\cdots=s_{n}=0.
\end{align}
With this we can reformulate (\ref{resn}),
\begin{align*}
f_n(q)=\sum_{n=s_0,s_1,\cdots, s_n}\prod_{i= 0}^{n-1} {s_i-s_{i+2}-1 \choose s_{i+1}-s_{i+2}}q^{{s_0\choose 2}-(s_1+s_2+\cdots+s_{n-1}+s_n)},
\end{align*}
where we sum over all sequences ${s_i}$ that satisfy (\ref{decreasing}).

Applying the change of variable $k_i=s_{i}-s_{i+1}$, $i=0,1,\cdots, n$, we obtain
\begin{align}
\label{fn}f_n(q)=\sum_{k_0+k_1+\cdots+k_{n}=n}\prod_{i=0}^{n} {k_i+k_{i+1}-1 \choose k_{i+1}}q^{{k_0+k_1+\cdots k_n\choose 2}-\sum_{i=0}^{n}{ik_i}}
\end{align}

To deal with this multisum notationally, we let
\begin{align}
\label{DefPq}P_q(k_0,k_1,k_2,\cdots, k_n)=\prod_{m=1}^{n}\left[{k_{m-1}+k_m-1 \atop k_m}\right]_q.
\end{align}
Notice that in (\ref{DefPq}), if some $k_i<0$, then $P_q(k_0,k_1,\cdots,k_n)=0$. If some $k_i=0$ and $P_q$ is non-vanishing, then we must have $k_i=k_{i+1}=\cdots =k_{n}=0$. We begin by proving a recurrence for $P_q(k_0,k_1,k_2,\cdots, k_n)$.
\begin{theorem}
Given any $0\leq \tau_0\leq k_0$, we have
\begin{align}
\label{Pq}P_q(k_0,k_1,k_2,\cdots, k_n)=\sum_{\tau_1,\tau_2,\cdots,\tau_n}
P_q(\tau_0,\tau_1,\cdots,\tau_n)P_q(k_0-\tau_0,k_1-\tau_1,\cdots,k_n-\tau_n)
q^{\sum_{i=0}^{n-1}\tau_i(k_{i+1}-\tau_{i+1})}.
\end{align}
Only terms with $0\leq \tau_i\leq k_i$ for $i=1,2,\cdots, n$, can be non-vanishing, so the right hand side of (\ref{Pq}) is a finite sum.
\end{theorem}
\begin{proof}
By the $q$-Vandermonde theorem, we have
\begin{align*}
\sum_{\tau_m}\left[{\tau_{m-1}+\tau_m-1 \atop \tau_m}\right]_q
\left[{k_{m-1}-\tau_{m-1}+k_m-\tau_m-1 \atop k_m-\tau_m}\right]_q
q^{\tau_{m-1}(k_m-\tau_m)}=\left[{k_{m-1}+k_m-1 \atop k_m}\right]_q
\end{align*}
for any $0\leq \tau_{m-1}\leq k_{m-1}$.

By repeated use of the above application of the $q$-Vandermonde theorem, we obtain
\begin{align*}
P_q(k_0,k_1,k_2,\cdots, k_n)
=&\sum_{\tau_1,\tau_2,\cdots,\tau_n}
\prod_{m=1}^{n}\left[{\tau_{m-1}+\tau_m-1 \atop \tau_m}\right]_q
\left[{k_{m-1}-\tau_{m-1}+k_m-\tau_m-1 \atop k_m-\tau_m}\right]_q
q^{\tau_{m-1}(k_m-\tau_m)}\\
=&\sum_{\tau_1,\tau_2,\cdots,\tau_n}P_q(\tau_0,\tau_1,\cdots,\tau_n)
P_q(k_0-\tau_0,k_1-\tau_1,\cdots,k_n-\tau_n)q^{\sum_{m=1}^{n}\tau_{m-1}(k_{m}-\tau_m)}.
\end{align*}
\end{proof}
Taking $q=1$, we obtain the following corollary:
\begin{corollary}
Given any $0\leq \tau_0\leq k_0$
\begin{align}
P_1(k_0,k_1,k_2,\cdots, k_n)=\sum_{\tau_1,\tau_2,\cdots,\tau_n}
P_1(\tau_0,\tau_1,\tau_2,\cdots,\tau_n)P_1(k_0-\tau_0,k_1-\tau_1,k_2-\tau_2,\cdots,k_n-\tau_n)
\end{align}
where
\begin{align*}
P_1(k_0,k_1,\cdots,k_n)=\prod_{m=1}^n{k_{m-1}+k_m-1\choose k_m}.
\end{align*}
\end{corollary}

In light of (\ref{fn}), define $C^\prime_n(q)$ as follows
\begin{align*}
q^{n\choose 2}C^\prime_n\left(\frac{1}{q}\right)=&\sum_{k_0+k_1+\cdots k_n=n}q^{{k_0+k_1+\cdots k_n\choose 2}-\sum_{i=0}^{n}ik_i}P_1(k_0,k_1,\cdots,k_n)\\
                 =&q^{n\choose 2}\sum_{k_0+k_1+\cdots k_m=n}q^{-\sum_{i=0}^{n}ik_i}P_1(k_0,k_1,\cdots,k_n).
\end{align*}
Consider the following expression:
\begin{align}
\label{sum}\sum_{l=0}^{n}q^{l}C^\prime_lC^\prime_{n-l}=\sum_{l=0}^{n}q^{l}
\sum_{\tau_0+\tau_1+\cdots\tau_n=l}q^{\sum_{i= 0}^{n}i\tau_i}P_1(\tau_0,\tau_1,\cdots,\tau_n)
\sum_{\mu_0+\mu_1+\cdots\mu_n=n-l}q^{-\sum_{i= 0}^{n}i\mu_i}P_1(\mu_0,\mu_1,\cdots,\mu_n).
\end{align}
Notice that $P(\tau_0,\tau_1,\cdots,\tau_n)=P(1,\tau_0,\tau_1,\cdots,\tau_{n})$ and $\sum_{i=0}^{n}(i+1)\tau_i=l+\sum_{i=0}^{n}i\tau_i$. We therefore have
\begin{align*}
\sum_{l=0}^{n}q^{l}C^\prime_lC^\prime_{n-l}
=&\sum_{l=0}^{n}\sum_{1+\tau_0+\tau_1+\cdots\tau_n=l+1}q^{-\sum_{i= 0}^{n}(i+1)\tau_i}P_1(\tau_0,\tau_1,\cdots,\tau_n)
\sum_{\mu_0+\mu_1+\cdots\mu_n=n-l}q^{-\sum_{i= 0}^{n}i\mu_i}P_1(\mu_0,\mu_1,\cdots,\mu_n)\\
=&\sum_{l=0}^{n}\sum_{1+\tau_0+\tau_1+\cdots\tau_n=l+1,\atop
\mu_0+\mu_1+\cdots\tau_n=n-l}q^{-\sum_{i= 0}^{n+1}i(\tau_{i-1}+\mu_i)}P_1(1,\tau_0,\tau_1,\cdots,\tau_n)
P_1(\mu_0,\mu_1,\cdots,\mu_n)\\
=&\sum_{k_0+k_1+k_2+\cdots k_{n+1}=n+1
}q^{-\sum_{i= 1}^{n+1}ik_i}\sum_{(\tau_0,\tau_1,\cdots,\tau_n)}P_1(1,\tau_0,\tau_1,\cdots\tau_n)
P_1(k_0-1,k_1-\tau_0,\cdots,k_{n+1}-\tau_n)\\
=&\sum_{k_0+k_1+k_2+\cdots=n+1
}q^{-\sum_{i=0}^{n+1}ik_i}P_1(k_0,k_1,k_2,\cdots,k_{n+1})=C^\prime_{n+1}
\end{align*}
Thus we have the following recurrence for $C^\prime_{n}(q)$:
\begin{align*}
C^\prime_{n+1}(q)=\sum_{l=0}^{n}q^{l}C^\prime_l(q)C^\prime_{n-l}(q)
\end{align*}

This is exactly the recurrence for (Cartlitz) $q$-Catalan numbers $C_n(q)$, so it must the case that $C^\prime_n(q) = C_n(q)$. This is what we wanted to show.

\subsection{A sequence of q,t-Catalan numbers} We may now define a natural sequence of $q,t$-Catalan numbers by
\begin{align}
\label{qt}q^{{n}\choose {2}}C_n\left(\frac{1}{q},t\right)=\sum_{\lambda \in \mathcal{P}, \atop n_\lambda \leq n} q^{|\lambda|} t^{h_{1,1}(\lambda)}
=&\sum_{\lambda \in \mathcal{P},\atop \lambda_1+l(\lambda) \leq n}\prod_{i \geq 1} \left[\lambda_i-\lambda_{i+2}+1 \atop \lambda_{i+1}-\lambda_{i+2}\right]_t q^{{\lambda_1\choose 2}+|\lambda|}.
\end{align}

These $q,t$-Catalan numbers bear some resemblance to the $q,t$-Catalan numbers studied by e.g. Haglund and Haiman, as can be seen from the following interpretation of their $q,t$-Catalan numbers:
\begin{align*}
q^{{n \choose 2}} \overline{C_n} \left(\frac{1}{q},t \right) = \sum_{\lambda \in \mathcal{P}, \atop n_\lambda \leq n} q^{|\lambda|} t^{h_{1,1}(\lambda) + h_0(\lambda)},
\end{align*}
where $h_0(\lambda)$ is defined to be the number of hook difference zero cells of $\lambda$, i.e. the number of cells with equal arm and leg lengths, and $\overline{C_n}$ is the $q,t$-Catalan number defined by Haglund and Haiman. Section 7 of \cite{LW} describes more about this interpretation of these $q,t$-Catalan numbers.

\begin{remark}
One may define a $q,t$-Catalan number, analogous to (\ref{qt}),
using $h_0$ instead of $h_{1,1}$. In fact, $h_0$ is constant on $2$-diagonal equivalence classes.
\end{remark}

\section{Refinement by $2$-core}\label{2core_section}

In the following section we consider the statistics $h_{\alpha,\beta}$ when $\alpha+\beta=2$ and their connection to the $2$-core of a partition. In \cite{BFN}, the authors proved that the statistics $h_{1,1}$ and $h_{2,0}$ are equidistributed over partitions of $n$. We conjecture they are equidistributed over partitions of $n$ with any given $2$-core, and give some partial results toward proving this.

\subsection{$m$-quotients and $m$-cores}

Given a partition $\lambda$, its \textbf{edge sequence} $M$ is a doubly infinite sequence of $0$'s and $1$'s defined by traversing the boundary shape occupied by complement of the Young diagram of $\lambda$ in the fourth quadrant with the upper left corner at the origin from bottom-left to top-right, with each $0$ representing a step to the north and each $1$ representing a step to the east.  We index the edge sequence by $\mathbb{Z}$ so that $0$-th term of the edge sequence corresponds to the step leaving the main diagonal of the Young diagram, and we denote the $i$-th term of the edge sequence $M$ by $M(i)$. Note that an edge sequence necessarily starts with infinitely many $0$'s and ends with infinitely many $1$'s, and the number of $1$'s before the $0$-th term (exclusive) equals  the number of $0$'s after the $0$-th term (inclusive).

Conversely, given any doubly infinite sequence $M$ of $0$'s and $1$'s that starts with infinitely many $0$'s and ends with infinitely many $1$'s, there is a unique integer $k$ such that defining $M'$ by $M'(i) = M(k+i)$ yields an edge sequence for some partition $\lambda$. We may therefore identify partitions with doubly infinite sequences of $1$'s and $0$'s that start with an infinite number of $0$'s and end with a infinite number of $1$'s, up to a shifting of the indices.

The \textbf{$m$-quotient} of partition $\lambda$ with edge sequence $M$ is an $m$-tuple of partitions $(\lambda_0, \lambda_1, ..., \lambda_{m-1})$ defined by $M_i(j) = M(m(j+k_i) + i)$ for some unique integer $k_i$ chosen so that $M_i$ is an edge sequence. We define $\lambda_i$ to be the unique partition corresponds to $M_i$. We call the $m$-tuple of integers $(k_0, k_1, ..., k_{m-1})$ the \textbf{$m$-shift} of $\lambda$. Note that we always have $\sum k_i = 0$.

Conversely, given any $m$-tuple of partitions and $m$-tuple of integers with zero sum, one may construct a partition with them as its $m$-quotient and $m$-shift by reversing the above construction. Thus partitions may be naturally identified with pairs of $m$-quotients and $m$-shifts.

\begin{example}
Consider the partition $\lambda=5,4,1,0,0,\cdots$, place its Young diagram in the fourth quadrant, then the corresponding edge sequence is $\cdots, 0,0,1,0,1,1,1,0,1,0,1,0,1,1,\cdots$. For the $2$-core and $2$-quotient of it, the edge sequences for the $2$-quotients are $\cdots, 0,0,1,0,0,1,1,\cdots$ and $\cdots,0,0,1,1,\cdots$, and the $2$-shift is $(2,-2)$.
\end{example}
We give a formula for computing the $m$-shift of a partition based on its Young diagram alone.

\begin{proposition}\label{mshift} Given a partition $\lambda$, let $N_i$ be the number of cells whose $(-1,1)$-label is congruent to $i$ mod $m$. Then the $m$-shift of $\lambda$ is $(N_0-N_1, N_1-N_2, ..., N_{m-1}-N_0)$. As a reminder, for a valid multigraph $M$, $N_i(M)$ are the number of north edges leaving $i$ if $i$ is a vertex of $M$ and $0$ otherwise. And here $M$ is the corresponding multigraph of $\lambda$, for simplicity we omit $M$.
\end{proposition}
\begin{proof} Let $n_r$ be the number of cells in $\lambda$ with $(-1,1)$-label $r$, and let $M$ be the edge sequence of $\lambda$. Then $n_r = |\{t \geq r | M(t) = 0\}|$ if $r\geq 0$ and $n_r = |\{t < r | M(t) = 1\}|$ if $r\leq 0$. (If we place the Young diagram in the fourth quadrant, $n_r$ is the number of cells on the diagonal $y=x-r$) Thus, for $i = 0, 1, ..., m-1$,
\begin{align*}
N_i - N_{i+1} &= \sum_{r \in \mathbb{Z}} (n_{rm+i} - n_{rm+i+1}) \\
&= \sum_{r \geq 0} (n_{rm+i} - n_{rm+i+1}) - \sum_{r < 0} (n_{rm+i+1} - n_{rm+i}) \\
&= |\{r\geq 0| M(rm+i) = 0\}| - |\{r < 0| M(rm+i) = 1\}| \\
&= k_i.
\end{align*}
\end{proof}

Note that removing an $lm$-hook in $\lambda$ corresponds to removing an $l$-hook in a component of its $m$-quotient, and leaves its $m$-shift unchanged.

The \textbf{$m$-core} of $\lambda$ is defined as the partition obtained by removing $m$-hooks in $\lambda$ until this is not possible. By the last paragraph, this is independent of the removal process, its $m$-shift is the same as that of $\lambda$, and its $m$-quotient consists of empty components. Thus $m$-cores naturally correspond to $m$-shifts, which means that partitions may be naturally identified with pairs of $m$-cores and $m$-quotients. Moreover, the size of a partition is equal to the sum of the size of its $m$ core and $m$ times the size of its $m$-quotient (defined as the sum of the sizes of its $m$ components).

For this section, we only need the case $m = 2$. (We will treat the general case in Section \ref{mcore_section}.) $2$-shifts are $(j, -j)$, and $2$-cores are staircase shapes $\{k, k-1, ..., 1\}$ for $k = \max\{-2j, 2j-1\}$, which have size $\binom{2j}{2}$. In particular, a $2$-core is uniquely determined by its size, which is always a triangular number. Proposition \ref{mshift} yields the following formula for computing the $2$-core of a partition.

\begin{lemma}Given partition $\mu$ with ($2$-)diagonal pattern $T(\mu)=\{1,2,\cdots,\lambda_1,\lambda_2,\cdots \lambda_k\}$, then the $2$-core of $\mu$ has size $\binom{2|\lambda|_a}{2}$, where $|\lambda|_a$ is the alternating sum
\begin{align*}
|\lambda|_a &= 1-2+3-\cdots +(-1)^{\lambda_1-1}\lambda_1+(-1)^{\lambda_1}\lambda_2+\cdots+(-1)^{\lambda_1+k-2}\lambda_k\\
&= (-1)^{\lambda_1-1}(\lceil\frac{\lambda_1}{2}\rceil+\sum_{i=2}^{k}(-1)^i\lambda_i)\\
&= (-1)^{l(\lambda^T)}\lfloor \frac{e(\lambda^T)-o(\lambda^T)}{2} \rfloor ,
\end{align*}
and $e(\mu), o(\mu)$ denote the number of even parts and odd parts of a partition, respectively.
\end{lemma}

\begin{remark}\label{2core} This means the partitions of $n$ with $2$-core size $\binom{2j}{2}$ are the partitions with diagonal pattern $\{1,2,\cdots,\lambda_1,\lambda_2,\cdots \lambda_k\}$ such that $\binom{\lambda_1}{2} + |\lambda| = n$ and $|\lambda|_a = j$.
\end{remark}

\begin{theorem}
\begin{align*}
\sum_{\lambda\in \mathcal{P}\atop |\lambda|_a=j}\prod_{i\geq 1}{\lambda_i-\lambda_{i+2}+1\choose \lambda_{i+1}-\lambda_{i+2}}q^{{\lambda_1\choose 2}+|\lambda|}=\frac{q^{2j\choose 2}}{\prod_{i\geq 1}(1-q^{2i})^2}.
\end{align*}
\end{theorem}
\begin{proof}By Theorem \ref{multisumthm} and Remark \ref{2core}, the left side is the generating function for partitions with $2$-core size $\binom{2j}{2}$, which is also the right side.
\end{proof}

We now propose a conjectural refinement of Theorem \ref{mainthm}. Let $\mathcal{P}_j(n)$ denote the set of partitions of $n$ with $2$-core size ${2j\choose 2}$, and let $\overline{\mathcal{P}_j}(n)$ denote the set of partitions with $2$-quotient size $n$ and $2$-core size ${2j \choose 2}$. Note that $\overline{\mathcal{P}_j}(n)= \mathcal{P}_j(2n+{2j \choose 2}).$
\begin{conjecture}\label{cj1}
We conjecture that the following three quantities are equidistributed over $\mathcal{P}_j(n)$:
\begin{enumerate}
\item $h_{1,1}(\mu)$,
\item $h_{2,0}(\mu)$,
\item $a_2(\mu)$.
\end{enumerate}
Moreover, we conjecture that the number of partitions $\mu$ in $\mathcal{P}_j(n)$ with $h_{1,1}(\mu)=m$ is $A(\frac{n-{2j\choose 2}}{2},m)$, where $A(n,m)$ is defined by the following generating function
\begin{align*}
\sum_{n=0}^{\infty}\sum_{ m=0}^{n} A(n,m)q^nt^m=\prod_{i\geq 1}\frac{1}{(1-q^i)(1-tq^i)}.
\end{align*}
\end{conjecture}

\begin{remark}
It is worth noting that $A(n,m)$ is sequence A103923 in the On-Line Encyclopedia of Integers Sequences \cite{OEIS}.
\end{remark}

\begin{remark}
It is worth noting that Corollary 1.3 of \cite{BFN}, which we restate as Theorem \ref{genbfn} in Section \ref{mcore_section}, implies
the three quantities are equidistributed over $\mathcal{P}(n)$, which is a weaker version of the above statement where we concern with $\mathcal{P}_j(n)$.
\end{remark}

\begin{remark}
It is natural to ask whether the above three quantities are equidistributed over diagonal classes. In general this is not the case as stated in Proposition \ref{k=0,1,2}.
\end{remark}

\subsection{Partial Results}

In this subsection, we give some partial results towards Conjecture \ref{cj1}. We prove that the statistics $h_{2,0}$ and $a_2$ are equidistributed over $\mathcal{P}_j(n)$ as desired, and that the statistics $h_{1,1}$ and $h_{2,0}$ are equidistributed over $\mathcal{P}_j(n)$ when $k=\max(2j,1-2j)$ is sufficiently large relative to $n$. We also show that $h_{2,0}$ is identically distributed over $\overline{\mathcal{P}_j}(n)$ as $j$ varies. One possible approach to establishing Conjecture \ref{cj1} completely would be to show that that same is true for $h_{1,1}$, for if this were true it would be possible to transfer our results for $h_{1, 1}$ from $\overline{\mathcal{P}_j}(n)$ with large $j$ to all $\overline{\mathcal{P}_j}(n)$, and so all $\mathcal{P}_j(n)$.

We begin by noting a special case of the conjecture that $h_{1,1}$ and $a_2$ are equidistributed $\mathcal{P}_j(n)$, and then give proofs of the above results.

\begin{proposition}\label{012prop}
For $r = 0, 1, 2$, we have
\begin{align*}
|\{\mu| \mu \in \mathcal{P}_j(n), h_{1,1}(\mu)=r\}| = |\{\mu| \mu \in \mathcal{P}_j(n), a_2(\mu)=r\}|
\end{align*}
\end{proposition}
\begin{proof} This follows from Proposition \ref{k=0,1,2} and the fact that partitions in a $2$-diagonal equivalence class have the same 2-core.
\end{proof}

\begin{theorem}\label{bij1}
The statistics $h_{2,0}$ and $a_2$ are equidistributed over $\mathcal{P}_j(n)$.
\end{theorem}
\begin{proof}
Let $k=\max(2j,1-2j)$. Given a partition $\lambda=(m^{b_m},(m-1)^{b_{m-1}},\cdots,2^{b_2},1^{b_1})$ of $n$ with $2$-core $\{k,k-1,\cdots,1\}$ (where $b_i$ is the number of occurences of $i$ in $\lambda$), we will construct a partition $\lambda'$ of $n$ with the same $2$-core as $\lambda$ so that $a_2(\lambda')=h_{2,0}(\lambda)$.

For every odd $b_i$, we move one part of size $i$ to the partition $\lambda'$. This will not change the value of $h_{2,0}(\lambda)$. The remaining parts of $\lambda$ all have even multiplicity, so the partition can be tiled by dominos and therefore has empty $2$-core. $\lambda'$ will therefore have the same $2$-core as $\lambda$, as it can be obtained by removing dominos from the original $\lambda$.

We then combine every pair of repeated parts of size $j$ in $\lambda$ into a single part of size $2j$, giving us exactly $h_{2,0}(\lambda)$ parts of even length. Taking the transpose of this gives us a partition $\mu$ tileable by dominos with $a_2(\mu)=h_{2,0}(\lambda)$. Adding these parts into $\lambda'$ will not change the $2$-core, so this will give us the desired partition.

This process is fully reversible. Given a partition $\lambda$ with $a_2(\lambda)=i$, we again move into a partition $\lambda'$ with distinct parts leaving behind only even multiplicities. Then if we take the transpose, we will have a partition with $i$ even parts. Splitting each part in half gives a partition $\mu$ with $h_{2,0}(\mu)=i$, and adding this to $\lambda'$ gives the desired inverse. So we have a bijection between partitions $\lambda$ of $n$ with $a_2(\lambda)=p$ and partitions with $p$ cells in $H_{2,0}$ that preserves $2$-core. This proves the theorem.
\end{proof}

\begin{figure}[H]
\begin{center}
\begin{picture}(10,12)
\linethickness{1.5pt}

\multiput(-8,0)(1,0){1}{\line(1,0){1}}
\multiput(-8,0)(1,0){2}{\line(0,1){1}}
\multiput(-8,1)(1,0){1}{\line(1,0){1}}
\multiput(-8,1)(1,0){2}{\line(0,1){1}}
\multiput(-8,2)(1,0){1}{\line(1,0){1}}
\multiput(-8,2)(1,0){2}{\line(0,1){1}}
\multiput(-8,3)(1,0){1}{\line(1,0){1}}
\multiput(-8,3)(1,0){2}{\line(0,1){1}}
\multiput(-8,4)(1,0){2}{\line(1,0){1}}
\multiput(-8,4)(1,0){3}{\line(0,1){1}}
\multiput(-8,5)(1,0){2}{\line(1,0){1}}
\multiput(-8,5)(1,0){3}{\line(0,1){1}}
\multiput(-8,6)(1,0){2}{\line(1,0){1}}
\multiput(-8,6)(1,0){3}{\line(0,1){1}}
\multiput(-8,7)(1,0){3}{\line(1,0){1}}
\multiput(-8,7)(1,0){4}{\line(0,1){1}}
\multiput(-8,8)(1,0){3}{\line(1,0){1}}
\multiput(-8,8)(1,0){4}{\line(0,1){1}}
\multiput(-8,9)(1,0){3}{\line(1,0){1}}
\multiput(-8,9)(1,0){4}{\line(0,1){1}}
\multiput(-8,10)(1,0){4}{\line(1,0){1}}
\multiput(-8,10)(1,0){5}{\line(0,1){1}}
\multiput(-8,11)(1,0){4}{\line(1,0){1}}

\multiput(-1,7)(1,0){2}{\line(1,0){1}}
\multiput(-1,7)(1,0){3}{\line(0,1){1}}
\multiput(-1,8)(1,0){2}{\line(1,0){1}}
\multiput(-1,8)(1,0){3}{\line(0,1){1}}
\multiput(-1,9)(1,0){4}{\line(1,0){1}}
\multiput(-1,9)(1,0){5}{\line(0,1){1}}
\multiput(-1,10)(1,0){6}{\line(1,0){1}}
\multiput(-1,10)(1,0){7}{\line(0,1){1}}
\multiput(-1,11)(1,0){6}{\line(1,0){1}}

\multiput(-1,0)(1,0){2}{\line(1,0){1}}
\multiput(-1,0)(1,0){3}{\line(0,1){1}}
\multiput(-1,1)(1,0){3}{\line(1,0){1}}
\multiput(-1,1)(1,0){4}{\line(0,1){1}}
\multiput(-1,2)(1,0){4}{\line(1,0){1}}
\multiput(-1,2)(1,0){5}{\line(0,1){1}}
\multiput(-1,3)(1,0){4}{\line(1,0){1}}

\multiput(8,5)(1,0){1}{\line(1,0){1}}
\multiput(8,5)(1,0){2}{\line(0,1){1}}
\multiput(8,6)(1,0){1}{\line(1,0){1}}
\multiput(8,6)(1,0){2}{\line(0,1){1}}
\multiput(8,7)(1,0){2}{\line(1,0){1}}
\multiput(8,7)(1,0){3}{\line(0,1){1}}
\multiput(8,8)(1,0){2}{\line(1,0){1}}
\multiput(8,8)(1,0){3}{\line(0,1){1}}
\multiput(8,9)(1,0){4}{\line(1,0){1}}
\multiput(8,9)(1,0){5}{\line(0,1){1}}
\multiput(8,10)(1,0){4}{\line(1,0){1}}
\multiput(8,10)(1,0){5}{\line(0,1){1}}
\multiput(8,11)(1,0){4}{\line(1,0){1}}

\multiput(15,2)(1,0){1}{\line(1,0){1}}
\multiput(15,2)(1,0){2}{\line(0,1){1}}
\multiput(15,3)(1,0){1}{\line(1,0){1}}
\multiput(15,3)(1,0){2}{\line(0,1){1}}
\multiput(15,4)(1,0){2}{\line(1,0){1}}
\multiput(15,4)(1,0){3}{\line(0,1){1}}
\multiput(15,5)(1,0){2}{\line(1,0){1}}
\multiput(15,5)(1,0){3}{\line(0,1){1}}
\multiput(15,6)(1,0){2}{\line(1,0){1}}
\multiput(15,6)(1,0){3}{\line(0,1){1}}
\multiput(15,7)(1,0){3}{\line(1,0){1}}
\multiput(15,7)(1,0){4}{\line(0,1){1}}
\multiput(15,8)(1,0){4}{\line(1,0){1}}
\multiput(15,8)(1,0){5}{\line(0,1){1}}
\multiput(15,9)(1,0){4}{\line(1,0){1}}
\multiput(15,9)(1,0){5}{\line(0,1){1}}
\multiput(15,10)(1,0){4}{\line(1,0){1}}
\multiput(15,10)(1,0){5}{\line(0,1){1}}
\multiput(15,11)(1,0){4}{\line(1,0){1}}

\color{red}{
\put(-4, 6){\vector(3,2){1.5}}
\put(-4, 3.5){\vector(3,-2){1.5}}
\put(5, 8){\vector(1,0){1.5}}
\put(12.5, 8){\vector(1,0){1.5}}
\put(8, 2){\vector(1,0){1.5}}
}

\end{picture}
\end{center}
\caption{The bijection of Theorem \ref{bij1}}.\label{graph1}
\end{figure}
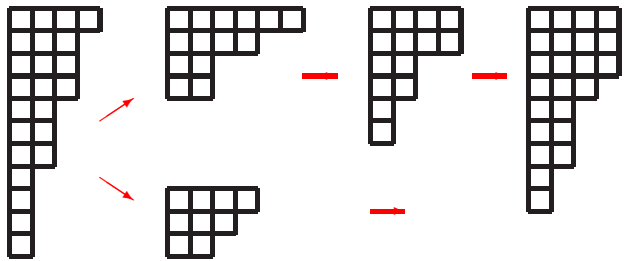

We now show that the statistics $h_{2,0}$ and $h_{1,1}$ are equidistributed over $\overline{\mathcal{P}_j}(n)$ for sufficiently large $k=\max(2j,1-2j)$. The following statement is a special case of Theorem \ref{bigcore}, which will be proved later.
\begin{proposition}\label{pp2}
Given any partition $\lambda$ with $2$-core $\{k,k-1,\cdots 1\}$ and $2$-quotient $(\lambda_0,\lambda_1)$ such that $k \geq |\lambda_0|+|\lambda_1|$, we have
\begin{enumerate}
\item if $2\mid k$, then $h_{1,1}(\lambda)=a(\lambda_1)$ and $h_{2,0}=a(\lambda_0)$.
\item if $2\nmid k$, then $h_{1,1}(\lambda)=a(\lambda_0)$ and $h_{2,0}=a(\lambda_1)$.
\end{enumerate}
This implies that $h_{1,1}$ and $h_{2,0}$ are equidistributed over $\overline{\mathcal{P}_j}(n)$ with $\max(2j,1-2j) \geq n$.
\end{proposition}

We now give two proofs of the fact that $h_{2,0}$ is identically distributed over $\overline{\mathcal{P}}_j(n)$ as $j$ varies over $\mathbb{Z}$.

\begin{proposition}\label{2coredist}
Let $\mathcal{P}_{j}$ be the set of all partitions with $2$-core size ${2j\choose 2}$. We then have the following generating function for partitions with distinct parts in $\mathcal{P}_j$ (or equivalently $2$-restricted partitions, by replacing $\lambda$ by its conjugate).
\begin{align*}
\sum_{\lambda\in \mathcal{P}_{j}\cap \mathcal{D}}q^{|\lambda|}=\frac{q^{2j\choose 2}}{\prod_{i\geq 1}(1-q^{2i})},
\end{align*}
where $\mathcal{D}$ is the set of partitions with distinct parts. This implies that the number of such partitions depends only on the size of the 2-quotient.
\end{proposition}
\begin{proof}[First proof] For each partition $\lambda$, let $\phi(\lambda)$ be the unique partition with distinct parts with that has $2$-diagonal pattern $\{1, 2, ..., \lambda^T_1, \lambda^T_2, ..., \lambda^T_k\}$. We know that $\phi : \mathcal{P} \mapsto \mathcal{D}$ is bijective, $|\phi(\lambda)| = \binom{l(\lambda)}{2} + |\lambda|$, and that the $2$-core size of $\phi(\lambda)$ is $\binom{2j}{2}$ for $j = (-1)^{l(\lambda)}\lfloor \frac{e(\lambda)-o(\lambda)}{2} \rfloor$. Thus the left side of the desired identity is equal to
\begin{align*}
\sum_{e(\lambda) - o(\lambda) = 2j, 1- 2j} q^{\binom{l(\lambda)}{2} + |\lambda|}.
\end{align*}
Dividing both sides by $q^{\binom{2j}{2}}$ and then replacing $q$ by $q^{1/2}$, we see that the desired identity is equivalent to
\begin{align*}
\sum_{e(\lambda) - o(\lambda) = 2j, 1- 2j} q^{\frac{|\lambda|+o(\lambda)(2e(\lambda)-1)}{2}} = \prod_{i\geq 1} \frac{1}{(1-q^i)}.
\end{align*}
Since the right side is the generating function for all partitions, it suffices to show that $\mathcal{P}(n)$, the set of partitions of $n$, is equinumerous with the set
\begin{align*}
A_{n, j} = \{ \lambda \text{ } \Big{|} \text{ } |\lambda|+o(\lambda)(2e(\lambda)-1) = 2n, e(\lambda) - o(\lambda) = 2j, 1- 2j\}.
\end{align*}

Given any partition $\lambda$, let $\alpha$ be the partition obtained by taking all odd parts of $\lambda$, adding $1$ to each part and then dividing each part by 2, and let $\beta$ be the partition obtained by taking all even parts of $\lambda$ and dividing each part by 2. Then $o(\lambda) = l(\alpha), e(\lambda) = l(\beta)$ and $|\lambda| = 2|\alpha| - l(\alpha) + 2|\beta|$, so $|\lambda|+o(\lambda)(2e(\lambda)-1) = 2(|\alpha|+|\beta|+l(\alpha)(l(\beta)-1))$. This gives a bijection between $A_{n, j}$ and the set of partition pairs
\begin{align*}
B_{n, j} = \{(\alpha, \beta) \text{ } \Big{|} \text{ } |\alpha|+|\beta|+l(\alpha)(l(\beta)-1) = n, l(\alpha) - l(\beta) = 2j, 1-2j\}
\end{align*}
To finish, we give a bijection between $\mathcal{P}(n)$ and $B_{n, j}$. Let $x = \max(2j, 1-2j)$. Given $\mu \in \mathcal{P}(n)$, choose the largest $k$ such that the rectangle of dimension $k \times (k+x-1)$ is contained in $\mu$. There are two types of partitions $\mu$ in $\mathcal{P}(n)$:

\begin{enumerate}[(a)]
\item\label{item1} $\mu$ not containing the rectangle of dimension $k \times (k+x)$, and
\item\label{item2} $\mu$ containing the rectangle of dimension $k \times (k+x)$.
\end{enumerate}

There are also two types of partition pairs in $B_{n, j}$:

\begin{enumerate}[(A)]
\item\label{itema} those that satisfy $l(\beta) - l(\alpha) = x$, and
\item\label{itemb} those that satisfy $l(\beta) - l(\alpha) = 1-x$.
\end{enumerate}

If $\mu \in \mathcal{P}(n)$ is of type \ref{item1}, let $\alpha$ be the portion of $\mu$ above row $k$ (exclusive) and to the right of column $k+x-1$ (inclusive), and $\beta$ be the conjugate of the portion of $\mu$ below row $k$ (inclusive). Then $(\alpha, \beta) \in B_{n, j}$ and is of type \ref{itema}. This construction is clearly reversible.

If $\mu \in \mathcal{P}(n)$ is of type \ref{item2}, let $\alpha$ be the conjugate of the portion of $\mu$ below row $k$ (inclusive) and to the left of column $k+x$ (exclusive), and $\beta$ be the portion of $\mu$ to the right of $k+x$ (inclusive). Then $(\alpha, \beta) \in B_{n, j}$ and is of type \ref{itemb}. Again, the construction is reversible.
\end{proof}
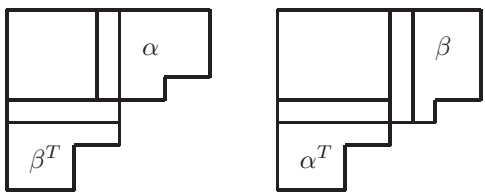
\begin{figure}
\begin{center}
\begin{picture}(15,10)
\linethickness{1pt}

\put(-2,0){\line(1,0){3}}
\put(-2,0){\line(0,1){3}}
\put(1,0){\line(0,1){2}}
\put(1,2){\line(1,0){2}}
\put(3,2){\line(0,1){1}}

\put(-2,3){\line(1,0){5}}
\put(-2,8){\line(1,0){5}}
\put(-2,3){\line(0,1){5}}
\put(3,3){\line(0,1){5}}

\put(-2,4){\line(1,0){5}}
\put(2,4){\line(0,1){4}}

\put(3,4){\line(1,0){2}}
\put(5,4){\line(0,1){1}}
\put(5,5){\line(1,0){2}}
\put(7,5){\line(0,1){3}}
\put(3,8){\line(1,0){4}}

\put(4,6){$\alpha$}
\put(-1,1){$\beta^T$}

\put(10,0){\line(1,0){3}}
\put(10,0){\line(0,1){3}}
\put(13,0){\line(0,1){2}}
\put(13,2){\line(1,0){2}}
\put(15,2){\line(0,1){1}}

\put(10,3){\line(1,0){5}}
\put(10,8){\line(1,0){5}}
\put(10,3){\line(0,1){5}}
\put(15,3){\line(0,1){5}}

\put(10,4){\line(1,0){5}}
\put(16,3){\line(0,1){5}}

\put(15,3){\line(1,0){2}}
\put(17,3){\line(0,1){1}}
\put(17,4){\line(1,0){2}}
\put(19,4){\line(0,1){4}}
\put(15,8){\line(1,0){4}}

\put(17,6){$\beta$}
\put(11,1){$\alpha^T$}

\end{picture}
\end{center}
\caption{The bijection between $\mathcal{P}(n)$ and $B_{n, j}$ of the first proof: type \ref{item1} is on the left and type \ref{item2} is on the right.}
\end{figure}

We give another proof of Proposition \ref{2coredist}.

\begin{proof}[Second proof]
Let $f_{a,b}(q)$ be the generating function for the number of partitions of $n$ with distinct parts, $2$-core $\{a,a-1,...1\}$ and largest part at most $b$. Note that for $f_{a,b}(q)$ to be nonzero, we must have $b>a$. Then we have the following lemma.

\begin{lemma}
\label{lgpt}$$f_{a,b}(q)=\left({b \atop \lfloor \frac{b-a}{2} \rfloor}\right)_{q^2}q^{a\choose{2}}$$

\end{lemma}
\begin{proof}
We prove this using induction on $b$. For the base case, $f_{0,0}=1$ and $f_{a,0}=0$ for $a\neq0$ as the only partition with largest part $0$ is the empty partition and it has empty $2$-core. For the inductive step, we prove the following recurrences for $f_{a,b}(q)$:
\begin{enumerate}
\item $f_{a,b}(q)=f_{a,b-1}+q^bf_{a-1,b-1}$ when $a>0,a\equiv b\mod2$
\item $f_{a,b}(q)=f_{a,b-1}+q^bf_{a+1,b-1}$ otherwise.
\end{enumerate}

Any partition with distinct parts, $2$-core $\{a,a-1,\cdots,1\}$ and largest part at most $b$ will either have a part of size $b$ or it will not. If it doesn't have such a part, then it will contribute to the generating function $f_{a,b-1}$. This gives the first term in each identity. If it does, then we can remove that part to get a partition of size $n-b$ with largest part at most $b-1$. Removing this part will remove $a$ from the $2$-core in the first case and add $a+1$ to the $2$-core in the second. This gives the second term in each recurrence.

By induction, we assume that the lemma holds for $f_{a',b-1}$ for all $a'$ and apply it to the right side of the identities. For identity (1), we get $$f_{a,b}(q)=\left({b-1 \atop \frac{b-a}{2}-1}\right)_{q^2}q^{\frac{a^2+a}{2}}+\left({b-1 \atop  \frac{b-a}{2}}\right)_{q^2}q^{\frac{a^2-a}{2}+b}.$$ Factoring out $q^{a\choose 2}$ and applying the identity $$\left({x \atop y}\right)_{q^2}=\left({x-1 \atop y-1}\right)_{q^2}+\left({x -1\atop y}\right)_{q^2}q^{2y}$$ gives the lemma when $a>0,a\equiv b\mod 2$.

The other case follows similarly using the identity $$\left({x \atop y}\right)_{q^2}=\left({x-1 \atop y}\right)_{q^2}+\left({x -1\atop y-1}\right)_{q^2}q^{2(x-y)}.$$
\end{proof}

When $b\geq n$, the generating function $f_{a,b}(q)$ will count all partitions of $n$ with distinct parts and $2$-core of height $a$. So the number of such partitions will be the coefficient of $q^{n-\frac{a^2+a}{2}}$ in $$\left(3n \atop \lfloor \frac{3n-a}{2} \rfloor \right)_{q^2}.$$ This is counted by the number of paths from the origin to $(\frac{3n-a}{2},\frac{3n+a}{2})$ with area $\frac{n-\frac{a^2+a}{2}}{2}$. These are just Young diagrams of the partitions of $\frac{n-{a\choose{2}}}{2}$. This is exactly the number counted by the generating function given in the proposition.
\end{proof}

Using the above proposition, we can obtain the following generating function.

\begin{theorem}\label{genhook20}
\begin{align}
\label{ge1}\sum_{\mu\in \mathcal{P}_j}t^{h_{2,0}(\mu)}q^{|\mu|}=
\frac{q^{2j\choose 2}}{\prod_{i\geq 1}(1-q^{2i})(1-tq^{2i})}
\end{align}
\end{theorem}
\begin{proof}
Given any partition $\mu\in \mathcal{P}_j$, remove all the pairs of repeated columns $2l_1,2l_2,\cdots, 2l_k$ of $\mu$, obtaining a $2$-restricted partition $\tilde{\mu}$. Note that the process of removing pairs of repeated columns is equivalent to moving some dominos from the original diagram, so it preserves the $2$-core, so $\tilde{\mu}\in \mathcal{P}_j$. Also note that the total number of repeated columns in a partition $\mu$ is equal $h_{2,0}(\mu)$(from the definition of $h_{2,0}$, each pair of repeated columns contribute one to $h_{2,0}$). Let $\lambda=\{2l_1,2l_2,\cdots,2l_k\}$. Notice that we have obtained a bijection between partitions in $\mathcal{P}_{j}$ and pairs of partitions $(\tilde{\mu},\lambda)$, which shows that the generating function for $h_{2,0}(\mu)$ is
\begin{align*}
\sum_{\mu\in \mathcal{P}_j}t^{h_{2,0}(\mu)}q^{|\mu|}
=\sum_{\tilde{\mu}\in \mathcal{P}_j\cap \mathcal{D}}q^{|\tilde{\mu}|}\sum_{\lambda\in \mathcal{P}_e}t^{l(\lambda)}q^{|\lambda|}=\frac{q^{2j\choose 2}}{\prod_{i\geq 1}(1-q^{2i})(1-tq^{2i})}
\end{align*}
where $\mathcal{P}_e$ is the set of partitions with even parts.
\end{proof}

\section{Generalization to the m-core case}\label{mcore_section}

Many of our results and conjectures can be generalized to cover the more general statistics $h_{\alpha,\beta}$ and make use of $m$-cores. In this section, we describe these results. We begin by stating the main theorem of \cite{BFN} in its full generality.

\begin{theorem}\label{genbfn} (\cite{BFN}, Corollary 1.3)
Let $\alpha$ and $\beta$ be non-negative integers which are not both zero. Then we have following generating function:
\begin{align*}
\sum_{\lambda \in \mathcal{P}} t^{h_{\alpha,\beta} (\lambda)} q^{|\lambda|}=\prod_{i\geq 1 \atop {(\alpha+\beta) \nmid 1}}\frac{1}{(1-q^{i})}\prod_{i \geq 1}\frac{1}{(1-tq^{(\alpha+\beta)i})}.
\end{align*}
\end{theorem}

\begin{remark}
The theorem was originally proved in \cite{BF}, but with the assumption that $\alpha$ and $\beta$ were coprime. The proof of this theorem made heavy use of algebraic geometry, and recall that our goal is to provide a combinatorial proof.
\end{remark}

We now generalize the Conjecture \ref{cj1} of the previous section as follows:

\begin{conjecture}\label{mcoreconj}
Given any $m$-core $\lambda_m$, let $\mathcal{P}_{\lambda_m}$ denote the set of partitions with $m$-core $\lambda_m$. Then
\begin{align}
\notag &\sum_{\mu\in \mathcal{P}_{\lambda_m}}t^{h_{\alpha,\beta}(\mu)}q^{|\mu|}=
\frac{q^{|\lambda_m|}}{\prod_{i\geq 1}(1-q^{mi})^{m-1}(1-tq^{mi})},
\end{align}
where $\alpha$ and $\beta$ are non-negative integers with sum $m$.
\end{conjecture}
\begin{remark}\label{mind}
This would imply that as $\lambda_m$ varies over all $m$-cores if $\alpha+\beta=m$, $h_{\alpha,\beta}$ is identically distributed over $\overline{\mathcal{P}_{\lambda_m}}(n) = \{\mu\in \mathcal{P}_{\lambda_m} \text{with  $m$-quotient size $n$} \}$ .
\end{remark}

The following theorem proves Remark \ref{mind} when the core is sufficiently large in a certain sense (more precisely, when the numbers in its $m$-shift are far apart).

\begin{theorem}\label{bigcore}
Given a partition $\lambda$ with edge sequence $M$, let $(k_0, k_1, ..., k_{m-1})$ be its $m$-shift and let $(\lambda_0,\lambda_1, ..., \lambda_{m-1})$ be its $m$-quotient. Let $M_i$ be the edge sequence of $\lambda_i$, and let $n = \max_{i\neq j}\{|\lambda_i|+|\lambda_j|)$. Recall that we may view each $M_i$ as a subsequence of $M$. Define $s_j = m k_j + j$ to be the index of the $0$-th term of $M_i$ in the edge sequence $M$. Reorder the $s_i$'s so that $s_{i_1} < s_{i_2} < ... < s_{i_m}$. If $l$ is an index such that $|s_{i_l} - s_{i_r}| \geq mn$ for all $r \neq l$, then $h_{l, m-l}(\lambda) = a(\lambda_{i_l})$.
\end{theorem}
\begin{proof}
We begin with some definitions.
We call a contiguous subsequence of an edge sequence a type $(a,l)$ \textbf{portion} if it starts with $1$, ends with $0$, and has $a$ $1$'s and $l$ $0$'s in between. Note that the cells of a partition with arm $a$ and leg $l$ correspond to the portions of type $(a, l)$ in its edge sequence. We call the maximal portion of an edge sequence its \textbf{essential} portion, before which only $0$s occur and after which only $1$s occur (which corresponds to the border of its Young diagram).

We have the following relation between the size of a partition and the essential portion of its edge sequence. Given any partition $\lambda$ with edge sequence $M$, define $p_1$ and $p_2$ so that the essential portion of $M$ starts at the $-p_1$-th term and ends at $p_2$-th term. We then have the following fact
\begin{align*}
\max\{p_1,p_2\}\leq |\lambda|.
\end{align*}

Let us now return to the original problem. Recall that $M$ is the edge sequence of $\lambda$, and $M_j$ is the edge sequence of $\lambda_j$, which we will think of as a subsequence embedded in $M$. Since $|s_{i_l} - s_{i_r}| \geq mn$ for all $r\neq l$, as a subsequence of $M$, the essential portion of $M_{i_r}$ is entirely to the left of the essential portion of $M_{i_l}$ if $r < l$ and is entirely to the right of the essential portion of $M_{i_l}$ if $r > l$.

Note that the cells in $H_{l, m-l}(\lambda)$ correspond to type $(nl-1, n(m-l))$ portions of $M$, for $n\in \mathbb{Z}_{> 0}$. If $P$ is a type $(nl-1, n (m-l))$ portion of $M$, then the first term $1$ and the last term $0$ are both in $M_{i_t}$ for some $t$. This means that $P$ and some portion of $M_{i_t}$ share the same staring and ending terms. Let $P_r = P \cap M_{i_r}$(Since $M_{i_t}$ does not consists of consecutive terms, $P_r$ is different from $P$). We have three cases:

\begin{enumerate}
\item $t < l$. Then for $r\in\{l,l+1,\cdots,m\}$, $P$ is entirely to the left of the essential portion of $M_{i_r}$ as a subsequence of $M$, so $P_r$ is a sequence of $n$ $0$'s. Thus, $P$ has at least $n(m-l+1)$ $0$'s not including the last $0$, which contradicts the fact that $P$ is of type $(nl-1, n (m-l))$.

\item $t > l$. Then for $r\in\{1,2,\cdots,l\}$, $P$ is entirely to the right of the essential portion of $M_{i_r}$, so $P_r$ is a sequence of $n$ $1$'s. Thus, $P$ has at least $nl$ 1s not including the first $1$, which contradicts the fact that $P$ is of type $(nl-1, n (m-l))$.

\item $t = l$. Then for $r\in\{1,2,\cdots,l-1\}$, $P_r$ is a sequence of $n$ $1$s, and for $r\in\{l+1,l+2,\cdots m\}$, $P_r$ is a sequence of $n$ $0$'s. Suppose $P_l$ is of type $(a, b)$, then $P$ is of type $(a+n(l-1), b+n(m-l))$. Since $P$ is of type $(n l-1, n (m-l))$, $a = n-1$ and $b = 0$, so $P_l$ is of type $(n-1, 0)$. Conversely, for any portion of type $(n-1, 0)$ of $M_{i_l}$, the portion of $M$ sharing the same starting and ending terms with it is of type $(nl-1, n(m-l))$.
\end{enumerate}

Therefore the portions in $M$ of type $(nl-1, n(m-l))$, which correspond to the cells of $H_{l, m-l}(\lambda)$, are in bijection with the portions in $M_{i_l}$ of type $(n-1, 0)$. This shows that $h_{l, m-l}(\lambda)$ equals the number of cells with leg $0$ in $\lambda_{i_l}$, which is $a(\lambda_{i_l})$.
\end{proof}

Given a partition $\mu$, a \textbf{$k$-cell} is a cell $v=(i,j)$ with $(1,-1)$-label equal to $k$ mod $m$, i.e. a cell $v=(i,j)$ so that $i-j\equiv k\mod p$. Define $w_k(\mu)$ to be the total number of $k$-cells in $\mu$. Define $\mathcal{A}_{k}(\mu)$ ($\mathcal{R}_k(\mu)$) to be the total number of \textbf{addable} $k$-cells (\textbf{removable} $k$-cells) for $k=0,1,2\cdots, p-1$. An addable $k$-cell $v$ is called \textbf{conormal} if the number of addable $k$-cells above the row of $v$ minus the number of removable $k$-cells above the row of $v$ is strictly greater than that for any higher addable $k$-cell (consider the Young diagram as placed in the fourth quadrant). A removable $k$-cell $v$ is called \textbf{conormal} if the number of removable $k$-cells below the row of $v$ minus the number of addable $k$-cells below the row of $v$ is strictly greater than that for any lower addable $k$-cell(consider the Young diagram as placed in the fourth quadrant).

\begin{example}
In the following Young diagram, take $m=2$, we have labelled by letter $C$ all the addable conormal $1$-cells.

\setlength{\unitlength}{0.5cm}
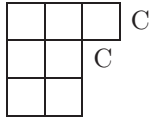
\begin{figure}[H]
\begin{center}
\begin{picture}(10,4)

\multiput(0,0)(1,0){2}{\line(1,0){1}}
\multiput(0,0)(1,0){3}{\line(0,1){1}}
\multiput(0,1)(1,0){2}{\line(1,0){1}}
\multiput(0,1)(1,0){3}{\line(0,1){1}}
\multiput(0,2)(1,0){3}{\line(1,0){1}}
\multiput(0,2)(1,0){4}{\line(0,1){1}}
\multiput(0,3)(1,0){3}{\line(1,0){1}}

\put(3.3,2.3){C}
\put(2.3,1.3){C}

\end{picture}
\end{center}
\caption{Example of addable conormal cells.}
\end{figure}
\end{example}

\begin{lemma}(\cite{FM} Proposition $2.30$)\label{coresize}
Given a partition $\mu$ and $w_k=w_k(\mu)$ for $k=0,1,2,\cdots,m-1$ as defined above, the size of the $m$-core $\mathcal{C}_m(\mu)$ of $\mu$ is
\begin{align*}
|\mathcal{C}_m(\mu)|=\frac{m}{2}\sum_{i=0}^{p-1}(w_{i+1}-w_{i})^2+\sum_{i=1}^{m-1}w_i-(m-1)w_0,
\end{align*}
where $w_{p}=w_{0}$.
\end{lemma}

\begin{lemma}\label{addremovedif}
Given a partition $\mu$ and $w_k=w_k(\mu)$ for $k=0,1,2,\cdots,m-1$ as defined above, we have the following relation between the number of addable $k$-cells and the number of removable $k$-cells of $\mu$:
\begin{align*}
\mathcal{A}_k(\mu)-\mathcal{R}_k(\mu)=
\left\{
\begin{array}{c}
w_{k+1}+w_{k-1}-2w_{k},\text{ if } k\neq 0\\
w_{p-1}+w_{1}-2w_{0}+1,\text{ if } k=0
\end{array}\right.
\end{align*}
where $w_m=w_0$. Moreover, if $\mu$ is a $m$-core, then either $\mathcal{A}_k(\mu)=0$ or $\mathcal{R}_k(\mu)=0$.
\end{lemma}
\begin{proof}
It is easy to check that the relation holds for the empty partition, imply induction on the number of cells, one can show that the relation holds for partitions with any size. Detailed proofs are left for readers as an exercise.
\end{proof}
\begin{remark}
Since $w_{k+1}(\mu)+w_{k-1}(\mu)-2w_{k}(\mu)$ depends only on the $m$-core of $\mu$, $\mathcal{A}_k(\mu)-\mathcal{R}_k(\mu)$ is constant over partitions with the same $m$-core for any $k\in \{0,1,2,\cdots,m-1\}$.
\end{remark}

Suppose we are given a partition $\mu$ and some $l\in\{0,1,2,\cdots, m-1\}$. We know from Lemma \ref{addremovedif} that $\mathcal{A}_l (\mu)-\mathcal{R}_l (\mu)=w_{l+1}(\mu)+w_{l-1}(\mu)-2w_{l}(\mu)+\chi(l=0)$. We obtain a new partition in the following way: if $\mathcal{A}_l(\mu)-\mathcal{R}_l(\mu)\geq 0$, then we add $\mathcal{A}(l)-\mathcal{R}(l)$ $k$-cells to $\mu$; else if $\mathcal{A}(l)-\mathcal{R}(l)< 0$, we remove $\mathcal{R}_l(\mu)-\mathcal{A}_l(\mu)$ $l$-cells from $\mu$. Call this new partition $\tilde{\mu}$.
\begin{lemma}\label{changecore}
The $m$-quotients of $\mu$ and $\tilde{\mu}$ have the same size.
\end{lemma}
\begin{proof}
Let $\tilde{w}_{k}=w_{k}(\tilde{\mu})$ and $w_{k}=w_{k}(\mu)$. From the construction of $\tilde{\mu}$, we have
\begin{align*}
\tilde{w}_k=\left\{
\begin{array}{cc}
w_k,&\quad k\neq l\\
w_{l+1}+w_{l-1}-w_{l}+\chi(l=0),&\quad k=l
\end{array}\right.
\end{align*}
From Lemma \ref{coresize},
\begin{align*}
\mathcal{C}_p(\tilde{\mu})
=&\frac{m}{2}\sum_{i=0}^{m-1}(\tilde{w}_{i+1}-\tilde{w}_{i})^2+\sum_{i=1}^{m-1}\tilde{w}_i-(m-1)\tilde{w}_0\\
=&\frac{m}{2}\sum_{i=0}^{m-1}(w_{i+1}-w_{i})^2+\sum_{i=1}^{m-1}w_i-(m-1)w_0+(w_{l+1}+w_{l-1}-2w_{l}+\chi(l=0))\\
=&\mathcal{C}_m(\mu)+w_{l+1}+w_{l-1}-2w_{l}+\chi(l=0).
\end{align*}
Therefore $\mu$ and $\tilde{\mu}$ have the same $m$-quotient size, as both the size of $\tilde{\mu}$ and the size of $\tilde{\mu}$'s $2$-core are $w_{l+1}+w_{l-1}-2w_{l}+\chi(l=0)$ greater than those of $\mu$.
\end{proof}

\begin{remark}
Suppose that we have an $m$-core $\lambda_m$. From Lemma \ref{addremovedif}, we have that either $\mathcal{A}_k(\lambda_m)=0$ or $\mathcal{R}_k(\lambda_m)=0$ for any $k\in\{0,1,2,\cdots,m-1\}$. Furthermore, Lemma \ref{changecore} tells us that we can add all the addable $k$-cells if $\mathcal{A}_k(\lambda_m)>0$, or, if $\mathcal{R}_k(\lambda_m)>0$, remove all the removable $k$-cells, and still be left with an $m$-core. In this way, we see that there is a natural poset structure on the set of $m$-cores which we may define as follows: given $m$-cores $\lambda_m$ and $\mu_m$, let $\lambda_m \lessdot \mu_m$ if and only if $\mu_m$ can be obtained from $\lambda_m$ by adding all the addable $k$-cells for some $k\in \{0,1,2,\cdots,m-1\}$. Then this poset of $m$-cores contains unique minimal element, i.e. the empty partition, and is a graded poset.
\end{remark}

We now show the above relation of adding (or removing) all the addable (or removable) $k$-cells of an $m$-core induces a bijection between $m$-restricted partitions with a given $m$-core, and those with a different $m$-core. This is a generalization of \cite{MOF}, where Fayers gives such a bijection in the $2$-core case.

\begin{lemma}\label{increasecore}
Given an $m$-core $\lambda_m$ such that $\mathcal{A}_k(\lambda_m)>0$, let $\mu_m$ be the $m$-core obtained from $\lambda_m$ by adding all the addable $k$-cells. Then the following is a bijection between $m$-restricted partitions with $m$-core $\lambda_m$ to $m$-restricted partitions with $m$-core $\mu_m$: given a partition $\lambda$ with $m$-core $\lambda_m$, map $\lambda$ to $\mu$ by adding the $\mathcal{A}_k(\lambda)-\mathcal{R}_k(\lambda)$ lowest conormal $k$-cells from $\mu$. Furthermore, this map preserves $m$-quotient size.
\end{lemma}
\begin{proof}
Using Lemma \ref{coresize}, we see that $\mu$ is a $m$-restricted partition with $m$-core $\mu_m$. From Lemma \ref{addremovedif}, $\mathcal{R}_k(\mu)-\mathcal{A}_k(\mu)=\mathcal{A}_k(\lambda)-\mathcal{R}_k(\lambda)$. Thus we can define an inverse map by mapping $\mu$ to $\lambda$ by removing the $\mathcal{R}_k(\mu)-\mathcal{A}_k(\mu)$ highest conormal $k$-cells from $\mu$.
\end{proof}
\begin{remark}
The above bijection preserves $h_{m,0}$, as it maps $m$-restricted partitions to $m$-restricted partitions.
\end{remark}

\begin{theorem}
Let $\mathcal{P}_{\lambda_m}$ be the set of partitions with $m$-core $\lambda_m$. We have the following generating functions for $m$-restricted partitions and $h_{m,0}$:
\begin{align}
\label{prestrict}\sum_{\mu\in \mathcal{P}_{\lambda_m}\cap \mathcal{R}_p}q^{|\mu|}&=\frac{q^{|\lambda_m|}}{\prod_{i\geq 1}(1-q^{mi})^{m-1}},\\
\label{hookp0}\sum_{\mu\in\mathcal{P}_{\lambda_m}}t^{h_{m,0}(\mu)}q^{|\mu|}&=\frac{q^{|\lambda_m|}}{\prod_{i\geq 1}(1-q^{mi})^{m-1}(1-tq^{mi})},
\end{align}
where $\mathcal{R}_m$ is the set of $m$-restricted partitions.
\end{theorem}
\begin{proof}
For (\ref{prestrict}), we see that there is a $m$-quotient size preserving bijection between $\mathcal{P}_{\lambda_m}\cap \mathcal{R}_m$ and $\mathcal{P}_{\mu_m}\cap \mathcal{R}_m$ for any $m$-core $\mu_m$ which may be obtained by repeated use of Lemma \ref{increasecore}. Since we may choose the $m$-core $\mu_m$ to satisfy the hypotheses of Theorem \ref{bigcore}, we get (\ref{prestrict}). The proof of (\ref{hookp0}) is similar to the proof of Theorem \ref{genhook20}.
\end{proof}

\begin{remark}
We know from Section \ref{multisum_section} that each $m$-diagonal equivalence class is uniquely determined by an $m$-restricted partition. Therefore the bijection in Lemma \ref{increasecore} induces a bijection which preserves $m$-quotient size between the $m$-diagonal equivalence classes associated with different $m$-cores.
\begin{example}
Consider the bijection from $2$-diagonal equivalence classes with $2$-core $\{1\}$ and $2$-quotient size $4$ to $2$-diagonal equivalence classes with $2$-core $\{2,1\}$ and $2$-quotient size $4$. Then the bijection of Theorem \ref{increasecore} is
\begin{align*}
[1,1,1,1,1,1,1,1,1]\Rightarrow&[2,1,1,1,1,1,1,1,1,1]\\
[2,2,1,1,1,1,1]    \Rightarrow& [2,2,2,1,1,1,1,1] \\
[3,2,1,1,1,1]       \Rightarrow&[3,3,2,1,1,1]\\
[2,2,2,2,1]        \Rightarrow& [2,2,2,2,2,1] \\
[3,2,2,1,1]        \Rightarrow& [4,3,2,1,1] \\
\end{align*}
where we use the unique $2$-restricted partition to represent each $2$-diagonal equivalence class.
\end{example}

We also know the generating function of $h_{1,1}$ in any given $2$-diagonal equivalence class from Corollary \ref{generatingfunction}. One might hope that this bijection between $2$-diagonal equivalence classes also preserves the generating function of $h_{1,1}$. However, this is not the case. In fact we obtain some non-trivial $q$-binomial identities, which we have no chance of solving in the general case, by equating the generating functions corresponding to different $2$-cores with given $2$-quotient size. From the example above, we obtain
\begin{align*}
&\left[2\atop 1\right]_q+\left[2\atop 1\right]_q
\left[2\atop 1\right]_q+\left[2\atop 1\right]_q+\left[3\atop 1\right]_q+\left[3\atop 1\right]_q\left[3\atop 1\right]_q\\
=&\left[2\atop 1\right]_q+\left[2\atop 1\right]_q\left[2\atop 1\right]_q+
\left[3\atop 2\right]_q\left[2\atop 1\right]_q+\left[3\atop 1\right]_q+\left[5\atop 1\right]_q.
\end{align*}

\end{remark}

\end{document}